\documentclass[10pt]{amsart}

\usepackage{pifont}
\usepackage[normalem]{ulem}

\usepackage[parfill]{parskip}

\usepackage{setspace}
\usepackage{amsmath, amsfonts, amsthm, amstext, amssymb, xspace, float, mathdots, stmaryrd}
\usepackage{eucal}
\usepackage{array}
\usepackage{amssymb,latexsym, mathtools}
\usepackage{xcolor}
\usepackage[shortlabels]{enumitem}
\usepackage{microtype}
\usepackage{longtable,booktabs}
\usepackage{tikz-cd}
\usetikzlibrary{quotes}
\usepackage{graphicx}
\usepackage{hyperref}
\usepackage{thmtools}
\usepackage[capitalise, nameinlink]{cleveref}
\hypersetup{
    colorlinks,
    linkcolor={purple!70!black},
    citecolor={blue!50!black},
    urlcolor={blue!80!black}
}
\usepackage{semtrans}

\theoremstyle{definition}
\numberwithin{equation}{section}
\newtheorem{thm}[equation]{Theorem}
\newtheorem{lemma}[equation]{Lemma}
\newtheorem{corollary}[equation]{Corollary}
\newtheorem{proposition}[equation]{Proposition}

\newtheorem{remark}[equation]{Remark}
\newtheorem{definition}[equation]{Definition}

\newtheorem{question}{Question}

\newtheorem{mainTheorem}{Main Theorem}

\usepackage{todonotes}

\def\c{\mathbb{C}}

\def\f{\mathbb{F}}

\def\p{\mathbb{P}}

\def\r{\mathbb{R}}
\def\z{\mathbb{Z}}

\def\ff{\mathbb{F}}
\def\zz{\mathbb{Z}}

\newcommand{\ull}[1]{\underline{#1}}

\def\colim{\operatorname{colim}}
\def\ker{\operatorname{ker}}
\def\coker{\operatorname{coker}}

\def\Th{\operatorname{Th}}
\def\RP{\mathbb{RP}}

\def\upperalphabet{\\A\\B\\C\\D\\E\\F\\G\\H\\I\\J\\K\\L\\M\\N\\O\\P\\Q\\R\\S\\T\\U\\V\\W\\X\\Y\\Z}
\def\loweralphabet{\\a\\b\\c\\d\\e\\f\\g\\h\\i\\j\\k\\l\\m\\n\\o\\p\\q\\r\\s\\t\\u\\v\\w\\x\\y\\z}
\def\lowergreekalphabet{\\\alpha\\\beta\\\gamma\\\delta\\\epsilon\\\zeta\\\eta\\\theta\\\kappa\\\lambda\\\mu\\\nu
    \\\xi\\\pi\\\rho\\\sigma\\\tau\\\upsilon\\\psi\\\chi\\\phi\\\omega}

\def\makecommands#1#2#3{
    \bgroup
    \def\tempcmdname##1{#1}
    \def\tempcmdbody##1{#2}
    \def\\##1{\expandafter\xdef\csname\tempcmdname{##1}\endcsname{\unexpanded\expandafter{\tempcmdbody{##1}}}}
    #3
    \egroup
}

\makecommands{#1mr}{\mathrm{#1}}{\upperalphabet}  
\makecommands{#1cal}{\mathcal{#1}}{\upperalphabet}  
\makecommands{#1#1}{\mathbb{#1}}{\upperalphabet} 
\makecommands{#1bar}{\overline{#1}}{\upperalphabet}
\makecommands{#1scr}{\mathscr{#1}}{\upperalphabet}
\makecommands{#1twee}{\widetilde{#1}}{\upperalphabet\loweralphabet} 
\makecommands{sf#1}{{\sf #1}}{\upperalphabet\loweralphabet}
\makecommands{#1frak}{\mathfrak{#1}}{\upperalphabet}  
\makeatletter
\makecommands{\expandafter\@gobble\string#1bar}{\overline{#1}}{\lowergreekalphabet}
\makecommands{\expandafter\@gobble\string#1twee}{\widetilde{#1}}{\lowergreekalphabet}
\makeatother

\DeclareMathOperator{\Map}{Map}

\DeclareMathOperator{\Top}{\mathcal{T}op}

\DeclareMathOperator{\res}{{\sf res}}

\author[M. Guo]{Manyi Guo}\address{University of Washington}\email{manyiguo@uw.edu}
\author[J. Morris]{Jackson Morris}\address{University of Washington}\email{jackmann@uw.edu}
\author[A. Waugh]{Alex Waugh}\address{University of Washington}\email{ajw48@uw.edu}
\author[A. J. Yang]{Albert Jinghui Yang}\address{University of Pennsylvania}\email{yangjh@sas.upenn.edu}

\title[Immersions of $\Cmr_2$-projective spaces via $\text{K}\r$-theory]{Immersions of $\Cmr_2$-projective spaces via $\text{K}\r$-theory}

\begin{document}

\begin{abstract}
    We compute the Atiyah Real K-theory of $\Cmr_2$-equivariant projective spaces and construct immersions of such spaces into multiples of the regular representation. These computations are made tractable by the recent geometric filtration of equivariant projective spaces due to Bhattacharya--Waugh--Zeng--Zou, together with a variant of the localized slice spectral sequence introduced by Meier--Shi--Zeng. As an immediate corollary of these computations, we obtain an equivariant analogue of James periodicity.
\end{abstract}

\maketitle

\tableofcontents

\section{Introduction}

The study of immersions of smooth manifolds into Euclidean spaces was a primary catalyst for the development of algebraic topology in the mid-twentieth century. What began as a question in differential geometry--determining the smallest dimension $n+k$ such that a given $n$-manifold immerses in $\r^{n+k}$--quickly evolved into a central problem in stable homotopy theory \cite{Whitney}. This transition stems from the observation that the existence of an immersion is equivalent to the existence of a corresponding bundle monomorphism \cite{HirschImmersions}, thereby translating the problem effectively into the language of characteristic classes and K-theory \cite{AtiyahImmersions}. 

Real projective spaces $\r\p^n$ have served as a litmus test for this problem. Since the stable normal bundle of an immersion $\r\p^n\looparrowright \r^{n+k}$ is stably equivalent to a multiple of the tautological line bundle, the study of immersions in this case is intimately tied to the stable behavior of the tautological line bundle. Atiyah showed that one can obtain an upper bound on $k$ by computing the order of the tautological line bundle over $\r\p^n$ in $\widetilde{\text{KO}}^0(\r\p^n)$. Building on work of James \cite{JamesJoin}, Fujii \cite{FujiiKO} identified this group as a cyclic group of order $2^{\phi(n)}$, where $\phi(n)$ is the {\bf James function} (see \Cref{thm:KO(RP^n)}). This provided the first significant refinement of Whitney's general immersion results and played a central role in Adams’ solution of the vector fields on spheres problem \cite{AdamsVFS}.

By reformulating immersion problems in terms of real K-theory,one is naturally led to consider the classifying space $\r\p^\infty\simeq \Bmr\Sigma_2$, obtained as the colimit of the tower
\[
    \begin{tikzcd}
        \r\p^0 \ar[r, hook, ""] & \r\p^1 \ar[r, hook, ""] & \r\p^2 \ar[r, hook, ""] & \cdots
    \end{tikzcd}
\]
This perspective is computationally powerful, as the K-theory of $\r\p^n$ may be viewed as a truncation of the K-theory of $\Bmr\Sigma_2$ for each $n$. 

In this paper, we study the analogous benchmark immersion problem for projectivizations of finite orthogonal $\Cmr_2$-representations. Since every such representation embeds into a multiple of the regular representation $\rho = 1 + \sigma$, we restrict attention to immersions of
\begin{equation} \label{defn: p(n rho)}
    \p(n\rho) \coloneqq S(n\rho)/\text{antipodes}
\end{equation}
into multiples of the $\Cmr_2$-regular representation. Our main immersion result is an equivariant refinement of Atiyah’s classical argument, using equivariant analogues of K-theory.

\begin{mainTheorem}[\Cref{thm:p(nrho) immersion existence}] \label{main:immersions}
    For each $n\geq 1$, there is a $\Cmr_2$-equivariant immersion 
    \[\p(n\rho)\looparrowright 2^{\phi(2n-1)}\rho,\] 
    where $\phi(n)$ is the number of integers $s$ such that $0 < s\leq n$ and $s\equiv 0, 1,2, 4\mod 8$. This $\phi(n)$ is the classical James function. 
\end{mainTheorem}

The equivariant analogue of Hirsch's work, relating the immersion problem to the existence of certain bundle morphisms, was established by Bierstone \cite{Bierstone}. The main results leading to \Cref{main:immersions} are equivariant analogues of work of Atiyah and Fujii, and may be summarized by the following objectives, where ${\rm KO}_{\Cmr_2}$ is the spectrum representing $\Cmr_2$-equivariant real K-theory \cite[Chapter XIV]{Alaska}:
\begin{enumerate}
    \item[1]
    \label{objective 1}Determine the order of the $\rho$-dimensional tautological bundle over $\p(n\rho)$, viewed as a class in $\text{KO}_{\Cmr_2}(\p(n\rho))$;
    \item[2]
    \label{objective 2}Show how to use the stable information obtained from $\text{KO}_{\Cmr_2}$-theory to deduce unstable information about vector bundles. 
\end{enumerate}

Computationally, we approach \eqref{objective 1} by first noticing that, as in the classical case, there is a directed system
\begin{equation}\label{filtration:p(n rho)}
    \begin{tikzcd}
        \p(\rho) \ar[r, hook, ""] & \p(2\rho) \ar[r, hook, ""] & \p(3\rho) \ar[r, hook, ""] & \cdots
    \end{tikzcd}
\end{equation}
whose colimit $\p(\infty\rho)$ may be identified with the classifying space $\Bmr_{\Cmr_2}\Sigma_2$ for $\Cmr_2$-equivariant principle $\Sigma_2$-bundles. Consequently, we primarily focus on this equivariant classifying space, restricting to $\p(n\rho)$ when necessary.

Then, instead of computing the $\Cmr_2$-equivariant K-theory of $\Bmr_{\Cmr_2}\Sigma_2$ entirely, we study its Atiyah Real K-theory \cite{Atiyah66}, represented by ${\rm K}\r$, and utilize the ``realification map":
\begin{equation}
\label{intro:c2realification}
    \Kmr\r^0(X)\to {\rm KO}^0_{\Cmr_2}(X)
\end{equation}
induced by the forgetful functor sending an Atiyah Real vector bundle to its underlying $\Cmr_2$-equivariant real vector bundle (see \Cref{sec:immersions}).

Objective \eqref{objective 2} concerns the passage from stable information in K-theory to unstable information about vector bundles. Classically, this follows from a cancellation theorem (see \cref{thm:ClassicalCancellation}), provided that the ranks of the vector bundles under consideration are sufficiently large. We establish an analogous cancellation theorem for $\Cmr_2$-equivariant vector bundles in \Cref{thm:C2Cancellation}.

Combining the computational input from the first objective with Bierstone's theorem \Cref{thm:BierstoneImmersion}, we obtain immersions into $2^n\rho$ (see \Cref{cor:P(nrho)Immersions}). We then sharpen this result by observing that the bundle under consideration over $\p(n\rho)$ is coinduced from the classical tautological line bundle, as in \eqref{eqn:pullbackconiducedBundle}, thereby obtaining the immersions stated in \Cref{main:immersions}. In particular, this shows that $\text{K}\r$ is as effective as $\text{KU}$ in detecting immersions of their respective real projective spaces. 

As a consequence of this refinement, we obtain a certain equivariant analogue of James periodicity for equivariant stunted projective spaces $\p_{k\rho}^{(n+k\rho)}$ (see \Cref{defn:stuntedProj}).

\begin{mainTheorem}[\Cref{thm:C2JamesPeriodicity}]
    For each $n>1$ and $k\in \z$, there is a stable equivalence
    \[
        \Sigma^{2^{\phi(2n-1)}\rho}\, \p_{k\rho}^{(k+n)\rho} \simeq \p_{(k+2^{\phi(2n-1)})\rho}^{(k + 2^{\phi(2n-1)} + n)\rho}, 
    \]
    where $\phi(n)$ is defined in \cref{thm:KO(RP^n)}.
\end{mainTheorem}

\subsection{Notation} \

For the rest of the paper, we use the following conventions.
\begin{itemize}
    \item $\Hmr \ull{\z}$ denote the $\Cmr_2$-equivariant Eilenberg-MacLane spectrum of the constant Mackey functor $\ull{\z}$.
    \item $\text{K}\r, \text{k}\r$ denote the Atiyah Real K-theory spectrum and its connective cover.
    \item $\text{KO}_{\Cmr_2}$ and $\text{KO}$ denote $\Cmr_2$-equivariant and classical real K-theory spectra. 
    \item $\text{KU}, \text{ku}$ denote the complex K-theory and its connective cover.
    \item $\p(n \rho)$ is projectivization of $n \rho$ for the $\Cmr_2$-regular representation $\rho$, as defined in \eqref{defn: p(n rho)}.
    \item We use $\Bmr_{\Cmr_2}\Sigma_2$ and $\p(\infty\rho)$ interchangeably. 
    \item We write $\star$ an arbitrary element of $\mathrm{RO}(\Cmr_2)$-grading and $*$ for an arbitrary integer.
\end{itemize}

\subsection{Outline of computations} \

The $\text{ku}$-homology of $\r\p^\infty$ can be computed via an Atiyah--Hirzebruch spectral sequence arising from the Whitehead filtration on $\text{ku}$. The associated graded of this filtration on $\text{ku}$ takes the form
\[
\text{gr}_n(\text{ku}) = \left\{\begin{array}{rl}
    \Sigma^{2m}\Hmr\z, & n = 2m \geq 0;\\
    0, & \text{else.}
\end{array} \right.
\]
Smashing with $\text{B}\Sigma_{2+}$ yields an Atiyah--Hirzebruch style spectral sequence with $\Emr_1$-page given by
\begin{equation}\label{eqn:kuSliceSS}
    \Emr_1 = \bigoplus_{m \geq0}\Sigma^{2m}\Hmr\z_*(\text{B}\Sigma_{2+}) \implies \text{ku}_*(\text{B}\Sigma_{2+}).
\end{equation}
This spectral sequence collapses on the $\Emr_1$-page for degree reason.

The equivariant analogue of the Whitehead filtration is the slice filtration. The associated graded of this filtration for $\text{k}\r$ takes the form (see \cite{Dugger})
\begin{equation}\label{eqn:kRSlice}
\text{P}^n_n(\text{k}\r) = \left\{\begin{array}{rl}
    \Sigma^{2m\rho}\Hmr\ull{\z}, & n = 2m \geq 0;\\
    0, & \text{else.}
\end{array} \right.
\end{equation}

By smashing with $\Bmr_{\Cmr_2}\Sigma_{2+}$, we obtain an equivariant Atiyah-–Hirzebruch–style spectral sequence analogous to \eqref{eqn:kuSliceSS}. Inspired by \cite{MSZ23}, we call this the {\bf augmented slice spectral sequence}:
\[\textbf{ASliceSS}_{\text{k}\r}(\Bmr_{\Cmr_2}\Sigma_2)\implies \text{k}\r_\star(\Bmr_{\Cmr_2}\Sigma_{2+}).\]
This spectral sequence is ``augmented" in the sense that it comes equipped with a preferred map to the slice spectral sequence for $\text{k}\r$, induced by the natural map $\Bmr_{\Cmr_2}\Sigma_{2+} \to S^0$, which collapses $\Bmr_{\Cmr_2}\Sigma_2$ to the non-basepoint.


Explicitly, $\textbf{ASliceSS}_{\text{k}\mathbb{R}}(\Bmr_{\Cmr_2}\Sigma_2)$ takes the form
\begin{equation}\label{eqn:kRSS}
    \Emr_1 = \bigoplus_{m \geq 0}\Sigma^{m\rho}\Hmr\underline{\z}_\star(\Bmr_{\Cmr_2}\Sigma_{2+}) \implies \text{k}\r_\star(\Bmr_{\Cmr_2}\Sigma_{2+}).
\end{equation}
We compute the $\Emr_1$-page of this spectral sequence by considering the spectral sequence obtained by smashing \eqref{filtration:p(n rho)} with $\Hmr\ull{\zz}$. Such a spectral sequence was recently studied in work of the third author with Bhattacharya--Zeng--Zou \cite{BWZZ25}, where they identify the associated graded explicitly. We then utilize this description to prove the following.

\begin{mainTheorem} [\cref{thm:HZHomologyBC2Sigma2}]\label{main:HZ}
    As $\Hmr\ull{\zz}_\star$-modules,
    \[
        \Hmr\ull{\zz}_\star(\Bmr_{\Cmr_2}\Sigma_{2+})\cong \Hmr\ull{\zz}_\star\langle \sfb_{0}\rangle \oplus \Cmr\langle\sfb_{m\rho+\sigma}\rangle_{m\geq 0} \oplus \Kmr\langle \sfb_{m\rho}\rangle_{m\geq 1}
    \]
    where $\Cmr$ and $\Kmr$ are the cokernel and kernel of the map $\Hmr\ull{\zz}_\star \xrightarrow{\cdot 2} \Hmr\ull{\zz}_\star$. The product structure is described completely in \Cref{thm:HZHomologyProductStructure}.
\end{mainTheorem}

Using the augmentation for \eqref{eqn:kRSS} and comparison with a simpler {\it localized} augmented slice spectral sequence (see \Cref{prop:kRPhiBC2Sigma2SS}), we are able to deduce all differentials in $\textbf{ASliceSS}_{\text{k}\r}(\Bmr_{\Cmr_2}\Sigma_2)$.

\begin{mainTheorem}[\Cref{thm:kRAssociatedGraded}]\label{main:kR}
    The $\Emr_\infty$-page of $\textbf{ASliceSS}_{\text{k}\r}(\Bmr_{\Cmr_2}\Sigma_{2})$ is
    \[
        \overline{\Dmr}\langle \sfb_0\rangle \oplus \overline{\Cmr}\langle \sfb_{n\rho+\sigma}\rangle_{n\geq 0}\oplus \overline{\Kmr}\langle \sfb_{n\rho}\rangle_{n\geq 1}.
    \]
    where $\overline{\Kmr}, \overline{\Cmr}$, and $\overline{\Dmr}$ are the $\Hmr\ull{\zz}_\star$-modules defined in \Cref{defn:assocGradedPieceskRSS}. The product structure on the $\mathrm{E}_\infty$-page (\cref{thm: ring str of kR BC2C2}) is induced from the product structure in \Cref{main:HZ} via the map $\text{k}\r \to \Hmr\ull{\mathbb{Z}}$.
\end{mainTheorem}



\subsection{Future Directions} \

Kitchloo and Wilson \cite{KitchlooWilson} use the second Real Johnson--Wilson theory to study non-immersions of $\r\p^n$. We expect that similar techniques can be applied to obtain non-immersion results for $\p(n\rho)$. As a first step towards computing the Real Johnson--Wilson homology of these spaces, we propose to compute their $\Bmr\Pmr_\r\langle n\rangle$-homology.

The differentials in the augmented slice spectral sequence in \Cref{main:kR} are completely determined by the augmentation, the localized augmented slice spectral sequence, and the Leibniz rule. Viewing connective Atiyah Real K-theory ${\rm k}\r$ as a model for $\Bmr\Pmr_\r\langle 1\rangle$, it is natural to ask if the $\Bmr\Pmr_\r\langle n\rangle$-homology of $\Bmr_{\Cmr_2}\Sigma_2$ can be computed in a similar fashion.

\begin{question}
    Are all the differentials in the augmented slice spectral sequence computing the $\text{BP}_\r\langle n\rangle$-homology of $\Bmr_{\Cmr_2}\Sigma_2$ determined by multiplicativity and the differentials in the slice spectral sequence for $\text{BP}_\r \langle n \rangle$?
\end{question}

One can also view our work as a starting point for the study of equivariant vector fields on equivariant spheres. Classically, the general strategy is to observe that the existence of $k-1$ linearly independent vector fields on $S^{n-1}$ is equivalent to the top cell of the stunted projective space $\RP_{n-k}^{n-1}$ splitting off. This is closely related to James periodicity, with the number of orthonormal vector fields on $S^{n-1}$ given by the maximum of all $k+1$ such that $\log_2\phi(n)$ divides $n$ \cite[Proposition 16.11.2]{Husemoller}. However, this number is determined by the rank of $\widetilde{\mathrm{ko}}^0(\RP^{n-1})$, rather than $\widetilde{\mathrm{ku}}^0(\RP^{n-1})$. Thus, to approach the $\Cmr_2$-equivariant vector fields problem, one is led to a direct and careful computation of the $\text{ko}_{\Cmr_2}$-homology of $\Bmr_{\Cmr_2}\Sigma_2$.

Finally, one can view our computation as a starting point for a more complete study of the homology of $\Cmr_2$-equivariant classifying spaces $\Bmr_{\Cmr_2}\Sigma_n$. Rather than studying such spaces individually, one observes that the homology of the disjoint union of these spaces for all $n\geq 1$ forms a \emph{Hopf ring}. This Hopf ring organizes the otherwise complicated homological computations into a larger, more tractable algebraic structure. Classically, the Hopf ring associated to the mod-2 homology of the disjoint union $\coprod \Bmr\Sigma_n$ was computed by Giusti--Salvatore--Sinha \cite{GiuSalSin12}. The study of $\text{C}_2$-equivariant Hopf rings was initiated by Petersen \cite{Pet24}, and the techniques she developed may be applied to any multiplicative $\text{C}_2$-equivariant homology theory.

\begin{question}
    Imitating the techniques of \cite{GiuSalSin12}, can one compute the Hopf ring structure on the $\Hmr\underline{\f}_2$ or $\Hmr\underline{\z}$ homology of the classifying spaces $\{\Bmr_{\Cmr_2}\Sigma_n$\}?
\end{question}

\subsection{Acknowledgments}

The authors would like to thank Juan Moreno, Yueshi Hou, Guchuan Li, and Shangjie Zhang for helpful discussions. The authors would furthermore like to thank XiaoLin Danny Shi for his enormous encouragement, support, and helpful insights throughout this project.


\section{The $\Hmr\underline{\z}$-homology of $\Cmr_2$-equivariant classifying space of $\Sigma_2$}

Classically, the skeletal filtration on $\r\p^\infty$ has associated graded pieces given by spheres. In contrast, the associated graded pieces of the filtration \eqref{filtration:p(n rho)} are not spheres, but can be described completely by lifting Atiyah's work \cite{AtiyahThom} to the $\Cmr_2$-equivariant world. These identifications have been carried out in \cite[Proposition 7.14]{BWZZ25}, where it is shown that there is a $\Cmr_2$-equivariant homeomorphism
\[
    \PP(n\rho)/\PP((n-1)\rho) \cong \Th((n-1)\xi_\rho),
\]
where 
\begin{equation}\label{eqn:xi_rho bundle}
    \xi_\rho := \left\{\begin{tikzcd} S(\rho\otimes \tau_2) \times_{\Sigma_2} (\rho_{\Cmr_2} \otimes \tau_2) \ar[d,""] \\ \p(\rho) \end{tikzcd}\right.
\end{equation}
and $\tau_2$ is the sign representation for $\Sigma_2$.

By \cite[Theorem 1.26]{BZ}, $\xi_\rho$ is $\Hmr\ull{\zz}$-orientable. Fixing such an orientation, smashing \eqref{filtration:p(n rho)} with $\Hmr\ull{\zz}$, one can use the corresponding Thom isomorphism to identify the $\Emr_1$-page of the associated spectral sequence as
\begin{equation}
\label{eqn:HZSS}
\begin{aligned}
    \Emr_1 = \bigoplus_{m \geq0}\Sigma^{m\rho} \Hmr\ull{\zz}_\star(\PP(\rho)_+) \implies \Hmr\ull{\zz}_\star(\Bmr_{\Cmr_2}\Sigma_{2+}).
\end{aligned}
\end{equation}
As $\PP(\rho)$ is $\Cmr_2$-equivariantly homeomorphic to $S^\sigma$, the $\Emr_1$-page of the spectral sequence can be further identified as an $\Hmr\ull{\zz}_\star$-module:
\[
    \bigoplus_{m \geq0} \Sigma^{m\rho}\Hmr\ull{\zz}_\star(\PP(\rho)_+) \cong \Hmr\ull{\zz}_\star\langle{\sf{b}}_{m\rho}, {\sf{b}}_{m\rho + \sigma}\rangle_{m\geq 0},
\]
where ${\rm{b}}_{{m\rho}}$ and ${\rm{b}}_{{m\rho+\sigma}}$ generate the filtration-$m$ part of the spectral sequence as an $\Hmr\ull{\zz}_\star$-module and have degrees indicated by their subscript.

$\Hmr\ull{\z}_\star$ is much more complicated than its nonequivariant counterpart. Indeed, we have that \cite{Greenlees17}
\[
\Hmr\ull{\z}_\star \cong \frac{\mathbb{Z}[u_{2\sigma}, a_\sigma]}{(2a_\sigma)} \oplus \bigoplus_{i \geq 0}\mathbb{Z}\left\{\frac{\theta_2}{u_{2\sigma}^i}\right\} \oplus \bigoplus_{\substack{j \geq 0 \\ k \geq 0}}\mathbb{F}_2\left\{\frac{\theta_3}{u_{2\sigma}^ja_\sigma^k} \right\},
\]
where $|u_{2\sigma}|=2-2\sigma$, $|a_\sigma| = -\sigma$, and $|\theta_n| = -n+n\sigma$. The first summand is often referred to as the ``positive cone", while the second two summand describe the ``negative cone". The notation of the negative cone is indicative of the action of the positive cone: $\theta_2$ is infinitely $u_{2\sigma}$-divisible, while $\theta_3$ is infinitely $u_{2\sigma}$ and $a_\sigma$-divisible. We also have that $\theta_2 \cdot \theta_3 =0$, implying that all products of elements in the negative cone are trivial. We depicted $\mathrm{H}\ull{\z}_\star$ in charts in \Cref{fig:HZ}. A black $\blacksquare$ denotes a $\zz$, and a \textcolor{blue}{blue} $\textcolor{blue}{\bullet}$ denotes a $\zz/2$. A \textcolor{blue}{blue} vertical line represents $a_\sigma$-multiplication.

\begin{figure}
    \begin{minipage}{1\textwidth}
    \centering
    \includegraphics[scale=.7]{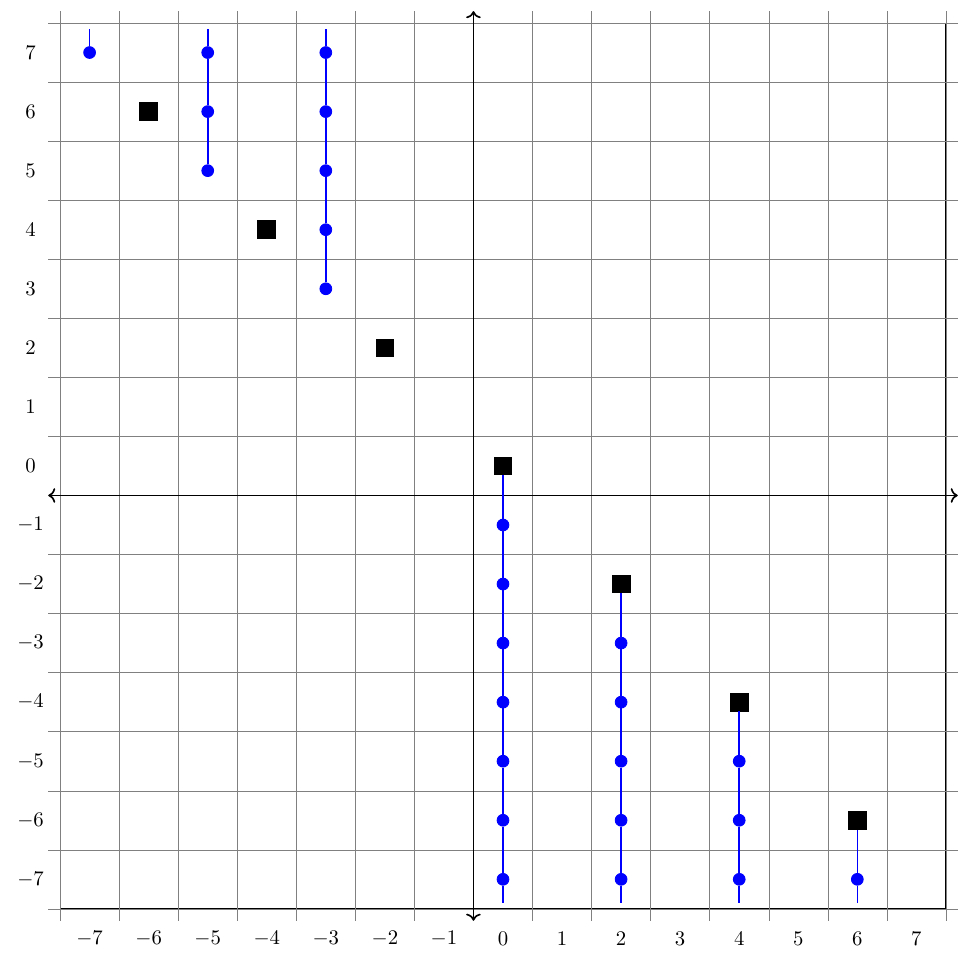}
    \caption{The coefficients $\Hmr\ull{\z}_\star$. }
    \label{fig:HZ}
    \end{minipage}
\end{figure}

\begin{lemma} \label{lem: d1-differential in HZSS}
    In \eqref{eqn:HZSS}, all of the $\rm{d}_1$-differentials are determined by 
    \[{\rm{d}}_1(\sfb_{m\rho}) = 2\cdot \sfb_{(m-1)\rho + \sigma}.\]
\end{lemma}

\begin{proof}
    We begin by comparing with the underlying spectral sequence:
    \begin{equation}\label{eqn:underlyingHZSS}
        \Emr_1 = \zz\langle\sfb_{2m}, \sfb_{2m+1}\rangle_{m\geq 0} \implies \Hmr\zz_*(\Bmr\Sigma_{2+}).
    \end{equation}
    The only nonzero differentials in this spectral sequence are ${\rm{d}}_1({\sfb_{2m}}) = 2\cdot \sfb_{2m-1}$ for $m \geq 0$. 
    Since the equivariant differentials $\rm{d}_r:\Emr_1^{n, V}\to \Emr_1^{n-r, V-1}$ are $\Hmr\ull{\zz}_\star$-module maps, it suffices to study their behavior on the $\sfb_{n\rho}$ and $\sfb_{m\rho + \sigma}$ classes. Comparing with the underlying spectral sequence \eqref{eqn:underlyingHZSS}, 
    \[
        \res^{\Cmr_2}_\sfe\circ\,\rm{d}_1(\sfb_{m\rho}) = \rm{d}_1\circ \res^{\Cmr_2}_\sfe(\sfb_{m\rho}) = 2\cdot \sfb_{2(m-1)+1}.
    \]
    As $\Emr_1^{m-1, m\rho-1}\cong \zz\langle \sfb_{(m-1)\rho + \sigma}\rangle$, it follows that $\rm{d}_1(\sfb_{m\rho}) = 2\cdot \sfb_{(m-1)\rho + \sigma}$. Similarly, 
    \[
        \res^{\Cmr_2}_\sfe\circ\,\rm{d}_1(\sfb_{m\rho+\sigma}) = \rm{d}_1\circ \res^{\Cmr_2}_\sfe(\sfb_{m\rho + \sigma}) = 0,
    \]
    and $\Emr_1^{m-1, m\rho + \sigma - 1} = 0$ implies $\rm{d}_1(\sfb_{m\rho + \sigma}) = 0$. 
\end{proof}

\begin{remark}
    The differentials in the underlying spectral sequence can be rewritten as ${\rm{d}}_1(\sfb_{2m}) = 2 \cdot \sfb_{2(m-1)+1},$ which more closely resembles the equivariant differential.
\end{remark}

The $\Emr_2$-page of \eqref{eqn:HZSS} can be described additively as
\[
    \Emr_2\cong \Hmr\ull{\zz}_\star\langle\sfb_{0}\rangle \oplus \Cmr\langle \sfb_{m\rho+\sigma}\rangle_{m\geq 0} \oplus \Kmr\langle \sfb_{m\rho}\rangle_{m\geq 1},
\]
where $\Cmr \coloneqq \coker\left(\begin{tikzcd}\Hmr\ull{\zz}_\star\ar[r, "2\cdot"] & \Hmr\ull{\z}_\star\end{tikzcd}\right)$ and $\Kmr \coloneqq \ker\left(\begin{tikzcd}\Hmr\ull{\zz}_\star\ar[r, "2\cdot"] & \Hmr\ull{\z}_\star\end{tikzcd}\right)$. 

We depict the $\Emr_1$-page in \Cref{fig:duckduckHZE1}, where the $\text{d}_1$-differential is drawn in red, and the $\Emr_2$-page in \Cref{fig:duckduckHZEinfty}. In both figures, we suppress the filtration degree.  
\begin{figure}
    \begin{minipage}{.45\textwidth}
    \centering
    \includegraphics[scale=.75]{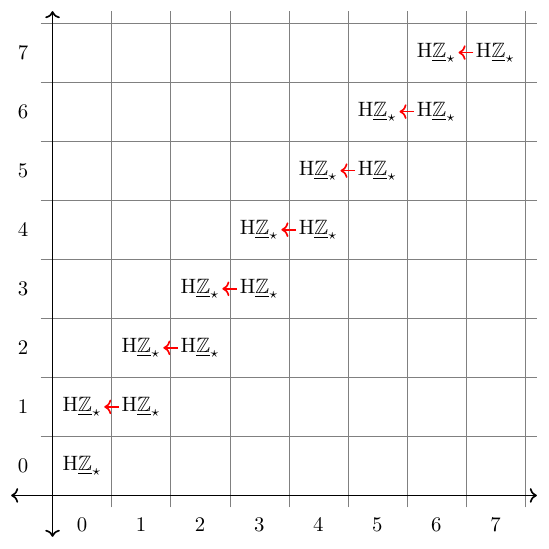}
    \caption{The $\Emr_1$-page of \eqref{eqn:HZSS}.}
    \label{fig:duckduckHZE1}
    \end{minipage}
    \begin{minipage}{.45\textwidth}
        \includegraphics[scale=.75]{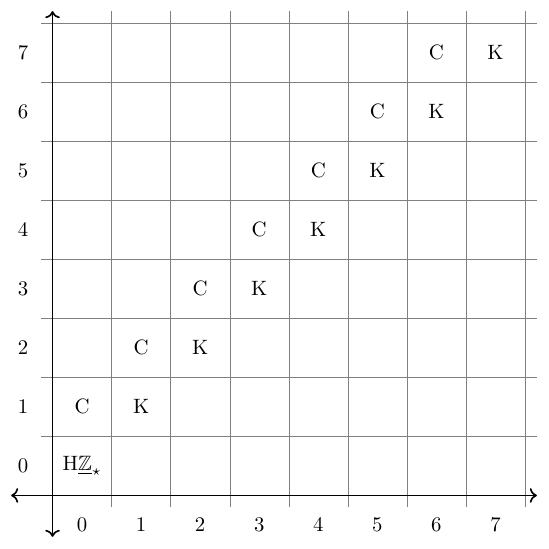}
        \caption{The $\Emr_2$-page of \eqref{eqn:HZSS}.}
        \label{fig:duckduckHZEinfty}
    \end{minipage}
\end{figure}

\begin{lemma} \label{lem: dr-differential in HZSS}
    For $r\geq 2$, all ${\rm{d}}_r$ differentials in \eqref{eqn:HZSS} vanish.
\end{lemma}

\begin{proof}
    Notice the copy of $\Cmr$ in filtration $m$ is generated by $1\cdot \sfb_{m\rho+\sigma}$ as an $\Hmr\ull{\zz}_\star$-module. The vanishing regions in $\Hmr\ull{\zz}_\star$ imply all higher differentials on this class vanish.

    The copy of $\Kmr$ in filtration $m$ is not generated by a single class as an $\Hmr\ull{\zz}_\star$-module, so we must determine potential differentials on classes of the form $a {\sf{b}}_{m\rho}$ for all $a \in \mathrm{H}\ull{\z}_\star$. The ring map $\Hmr\ull{\z}\to \Hmr\ull{\ff}_2$ induces a map of spectral sequences arising from \eqref{filtration:p(n rho)}. The $\Hmr\ull{\ff}_2$-spectral sequence \cite{NickG} collapses on the $\Emr_1$-page. The composition 
    \[
    \begin{tikzcd}
        \Kmr \ar[r, ""] & \Hmr\ull{\zz}_\star \ar[r, ""] & \Hmr\ull{\ff}_{2\star}
    \end{tikzcd}
    \]
    is an injection, as can be seen by \Cref{fig:K}. It follows that the $\mathrm{d}_r$-differential on any $\alpha \sfb_{m\rho}$ with $\alpha \in \Kmr$ is trivial.
\end{proof}

Therefore, the spectral sequence \eqref{eqn:HZSS} collapses at $\Emr_2 = \Emr_\infty$. As $\Kmr$ is a projective $\Hmr\ull{\zz}_\star$-module, there are no hidden extensions involving the $\Kmr$-terms. By comparing with underlying, there are no hidden extensions among the $\Cmr$-terms. So combining \cref{lem: d1-differential in HZSS} and \cref{lem: dr-differential in HZSS}, we deduce the following result.

\begin{thm}\label{thm:HZHomologyBC2Sigma2}
    As an $\Hmr\ull{\zz}_\star$-module, 
    \[
        \Hmr\ull{\zz}_\star(\Bmr_{\Cmr_2}\Sigma_{2+})\cong \Hmr\ull{\zz}_\star\langle \sfb_{0}\rangle \oplus \Cmr\langle \sfb_{m\rho + \sigma}\rangle_{m\geq 0} \oplus \Kmr\langle \sfb_{m\rho}\rangle_{m\geq 1}.
    \]
\end{thm}

As a result of the intricate structure of $\mathrm{H}\ull{\z}_\star$,  while \Cref{fig:duckduckHZEinfty} provides a concise depiction of the $\Emr_\infty$-page of \eqref{eqn:HZSS}, it also suppresses the underlying complexity of this computation. These details will be necessary as we compute with $\text{k}\mathbb{R}$. For the reader’s convenience, we include more detailed charts describing $\text{K}$ and $\text{C}$ in \Cref{fig:K} and \Cref{fig:C}, respectively.

\begin{figure}
    \begin{minipage}{.45\textwidth}
    \centering
    \includegraphics[scale=.4]{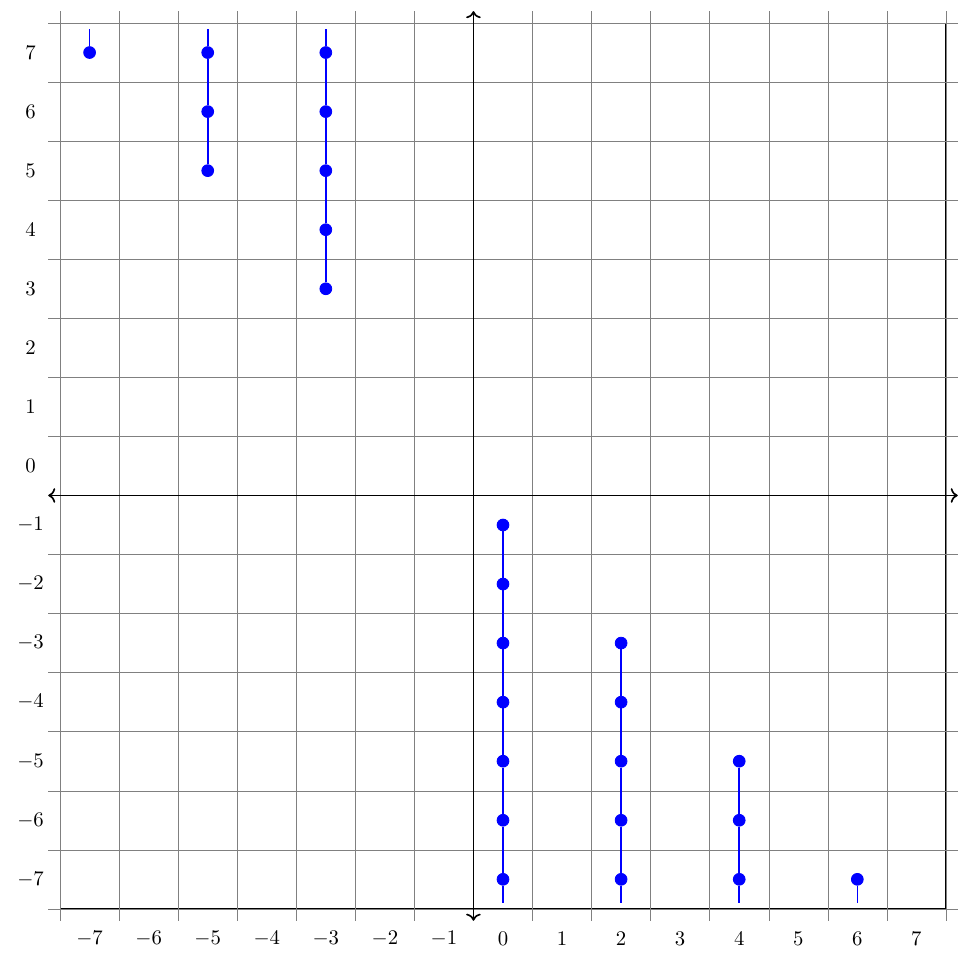}
    \caption{The $\Hmr\ull{\z}_\star$-module \\ $\text{K} \coloneq \text{ker}(\Hmr\ull{\zz}_\star \xrightarrow{2}\Hmr\ull{\zz}_\star)$.}
    \label{fig:K}
    \end{minipage}
    \begin{minipage}{.45\textwidth}
        \includegraphics[scale=.4]{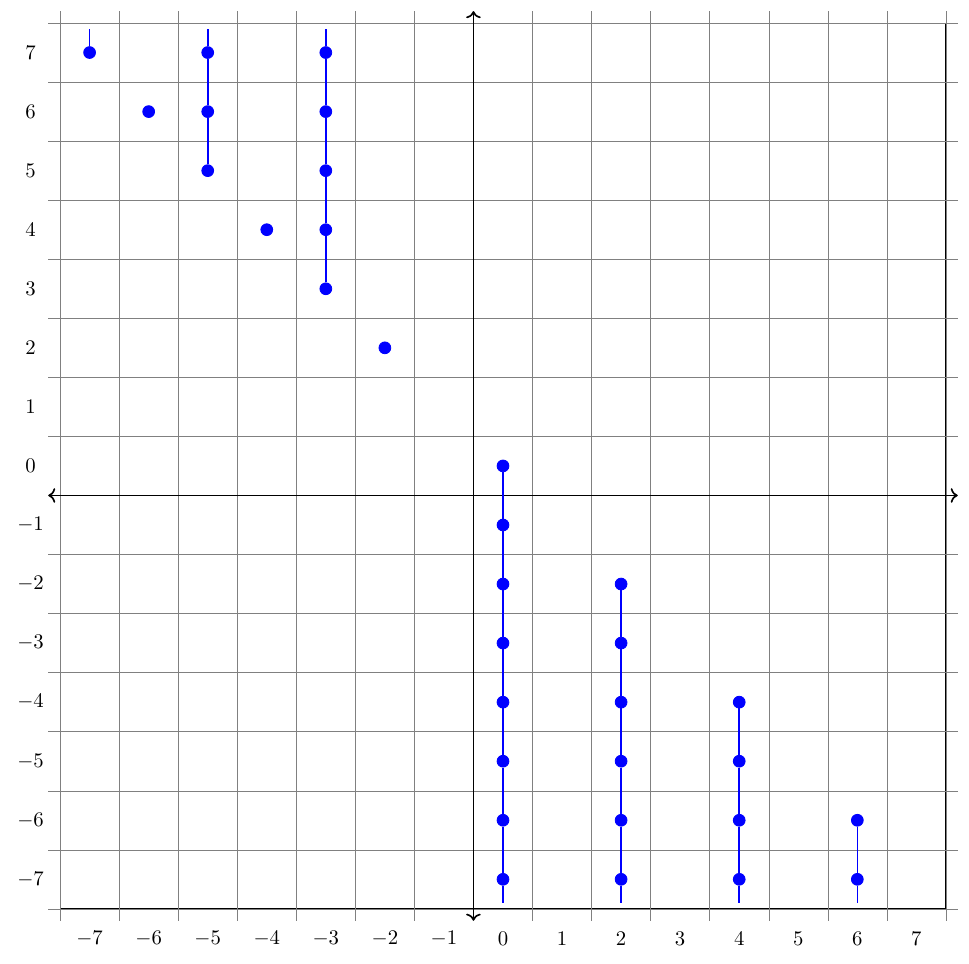}
        \caption{The $\Hmr\ull{\z}_\star$-module \\ $\text{C} \coloneq \text{coker}(\Hmr\ull{\zz}_\star \xrightarrow{2}\Hmr\ull{\zz}_\star)$.}
        \label{fig:C}
    \end{minipage}
\end{figure}

The filtration \eqref{filtration:p(n rho)} descends to a filtration on the $\text{C}_2$-equivariant projective space $\mathbb{P}(n\rho)$ by truncation. Smashing with $\Hmr\ull{\z}$ yields a spectral sequence of the form
\begin{equation}
\label{ss:HZPnrho}
    \Emr_1 = \bigoplus_{m =0}^n\Sigma^{m\rho}\Hmr\ull{\z}_\star(\mathbb{P}(\rho)_+) \implies \Hmr\ull{\z}_\star(\mathbb{P}(n\rho)_+).
\end{equation}
The following is immediate.
\begin{corollary}
    The spectral sequence \eqref{ss:HZPnrho} is determined by the differential
    \[\text{d}_1({\sf{b}}_{m\rho}) = 2 \cdot {\sf{b}}_{(m-1)\rho + \sigma}\]
    and collapses at the $\Emr_2$-page. As an $\Hmr\ull{\z}_\star$-module, 
    \[\Hmr\ull{\z}_\star(\mathbb{P}(n\rho)_+) \cong \Hmr\ull{\z}_\star\langle {\sf{b}}_0\rangle \oplus \text{C}\langle{\sf{b}}_{m\rho+\sigma}\rangle_{m=0}^n \oplus \text{K}\langle {\sf{b}}_{m\rho} \rangle_{m=1}^n.\]
\end{corollary}

\subsection{The Pontryagin Ring Structure} \

As the space $\Bmr_{\Cmr_2}\Sigma_2$ is a $\Cmr_2$-equivariant $\Hmr$-space \cite{Alaska}, its $\Hmr\ull{\zz}$-homology groups form a Pontryagin ring. In this section, we compute this ring structure by first determining the ring structure on its $\Hmr\ull{\ff}_2$-homology.

Hu and Kriz \cite{HK01} computed
\begin{equation}\label{eqn:F2CohomologyBC2Sigma2}
    (\Hmr\underline{\f}_2)^{\star}(\Bmr_{\Cmr_2} \Sigma_2) = \frac{\Hmr \underline{\f}_2^{\star} \llbracket \sfe_{\Hmr\ull{\ff}_2} \rrbracket [y]}{y^2 = a_{\sigma} y+ u_{\sigma} \sfe},
\end{equation}
which is a free $\Hmr\ull{\ff}_2^\star$-module. Here the class $\sfe_{\Hmr\ull{\ff}_2}$ is an $\Hmr\underline{\f}_2$-Euler class of the $\rho$-dimensional $\Cmr_2$-equivariant vector bundle $\xi$ used in studying the filtration \eqref{filtration:p(n rho)}. After dualizing, we obtain
\begin{equation} \label{eqn: HF2(BC2C2) ring str}
    (\Hmr\underline{\f}_2)_{\star}(\Bmr_{\Cmr_2}\Sigma_{2+}) = \Hmr\f_2 \langle \sfb_{m \rho}, \sfb_{m \rho + \sigma}\rangle_{m \geq 0},
\end{equation}
where $\sfb_{m\rho} = (\sfe_{\Hmr\ull{\ff}_2}^n)^{\vee}$ and $\sfb_{m \rho + \sigma} = (\sfe_{\Hmr\ull{\ff}_2}^n \cdot y)^{\vee}$. For degree reasons, one immediately observes 
\[
\sfb_{n\rho + \sigma}\cdot \sfb_{m\rho + \sigma} = 0.
\]

\begin{remark}
    This is perhaps surprising because the restrictions of these classes to underlying are $\sfb_{2n + 1}$ and $\sfb_{2m+1}$, respectively. If $n\neq m$, the product of the underlying classes is nonzero. This does not lead to a contradiction because the restriction map in Bredon {\it homology} is not, in general, a ring homomorphism. In particular, this means that one cannot immediately deduce the product structure from the underlying product structure. Nevertheless, we can still imitate the classical arguments. 
\end{remark}
By comparison with the underlying, it follows that
\[
\Delta(\sfe_{\Hmr\ull{\ff}_2}) = \sfe_{\Hmr\ull{\ff}_2}\otimes 1 + 1\otimes \sfe_{\Hmr\ull{\ff}_2}\in \Hmr\ull{\ff}_2^\star(\Bmr_{\Cmr_2}\Sigma_{2+})^{\otimes 2},
\]
where $\Delta$ is the comultiplication induced by the multiplication $\Bmr_{\Cmr_2}\Sigma_2\times \Bmr_{\Cmr_2}\Sigma_2\to \Bmr_{\Cmr_2}\Sigma_2$.
Consequently, $\sfb_{n\rho}\cdot \sfb_{m\rho} = (1-\delta_m^n)\sfb_{(n+m)\rho}$, where $\delta_m^n$ is the Kronecker delta. It remains to study $\sfb_{m\rho + \sigma}\cdot \sfb_{n\rho}$. 

\begin{lemma}
    For any $n\geq 0$, $\sfb_{\sigma} \cdot \sfb_{n\rho} = \sfb_{n\rho + \sigma}$. 
\end{lemma}

\begin{proof}
    The class $y$ can be chosen such that $\beta(y) = \sfe_{\Hmr\ull{\ff}_2}$, where $\beta$ is connecting homomorphism for the cofiber sequence
    \[
        \begin{tikzcd}
            \Hmr\ull{\ff}_2 \ar[r, ""] & \Hmr\ull{\ff}_4 \ar[r, ""] & \Hmr\ull{\ff}_2
        \end{tikzcd}
    \]
    By the naturality of the Bockstein, 
    \begin{align*}
        \begin{split}
            \beta\circ \Delta^*(y) &= \Delta^*\circ \beta(y)
                = \Delta^*(\sfe_{\Hmr\ull{\ff}_2})
                = \sfe_{\Hmr\ull{\ff}_2}\otimes 1 + 1\otimes \sfe_{\Hmr\ull{\ff}_2}.
        \end{split}
    \end{align*}
    For degree reasons, it follows that $\Delta^*(y) = y\otimes 1 + 1\otimes y$. Thus
    \begin{align*}
        \begin{split}
            (\sfb_{\sigma}\cdot \sfb_{n\rho})(\sfe_{\Hmr\ull{\ff}_2}^n\cdot y) &= (\sfb_\sigma\otimes \sfb_{n\rho})(\Delta^*(\sfe_{\Hmr\ull{\ff}_2}^n)\Delta^*(y))\\
            &= (\sfb_\sigma\otimes \sfb_{n\rho})((\sfe_{\Hmr\ull{\ff}_2}\otimes 1 + 1\otimes \sfe_{\Hmr\ull{\ff}_2})^n(y\otimes 1 + 1\otimes y))
            = 1.
        \end{split}
    \end{align*}
\end{proof}

\begin{corollary}\label{cor:HF2HomologyProductStructure}
    The ring structure of \eqref{eqn: HF2(BC2C2) ring str} is determined by the following relations:
    \begin{itemize}
        \item $\sfb_{n \rho+ \sigma} \cdot \sfb_{m \rho + \sigma} = 0$, for $m, n \geq 0$;
        \item $\sfb_{n \rho} \cdot \sfb_{m \rho} = (1-\delta_n^m) \sfb_{(m+n) \rho}$, where $\delta^m_n$ is the Kronecker delta symbol;
        \item $\sfb_{n\rho} \cdot \sfb_{\sigma} = \sfb_{n\rho + \sigma}$.
    \end{itemize}
\end{corollary}

Nonequivariantly, the ring structure on the $\Hmr\zz$-homology of $\Bmr\Sigma_2$ is trivial for degree reasons. Equivariantly, this is not the case due to the presence of the $\Kmr$-factors in \Cref{thm:HZHomologyBC2Sigma2}. 

\begin{thm}\label{thm:HZHomologyProductStructure}
    The isomorphism in \Cref{thm:HZHomologyBC2Sigma2} is a ring homomorphism where
    \begin{itemize}
        \item $\sfb_{n \rho+ \sigma} \cdot \sfb_{m \rho + \sigma} = 0$, for $m, n \geq 0$;
        \item $(\alpha \sfb_{n \rho}) \cdot (\beta\sfb_{m \rho}) = (1-\delta_n^m) \alpha\beta\cdot \sfb_{(m+n) \rho}$, where $\delta^m_n$ is the Kronecker delta, and $\alpha,\beta\in \Kmr$;
        \item $(\alpha\cdot\sfb_{n\rho}) \cdot \sfb_{m\rho + \sigma} = \alpha\cdot\sfb_{(n+m)\rho + \sigma}$, where $\alpha\in \Kmr$.
    \end{itemize}
\end{thm}

\begin{proof}
    The cofiber sequence 
    \[
        \begin{tikzcd}
            \Hmr\ull{\zz} \ar[r, "2"] & \Hmr\ull{\zz} \ar[r, ""] & \Hmr\ull{\ff}_2
        \end{tikzcd}
    \]
    induces a long exact sequence in the homology of $\Bmr_{\Cmr_2}\Sigma_{2+}$. On all classes except $\Hmr\ull{\zz}_\star$-multiples of $\sfb_0$, the multiplication by $2$ map is trivial. The product of integral classes then follows immediately from \Cref{cor:HF2HomologyProductStructure}. Since $\sfb_0$ is the multiplicative unit, the result follows. 
\end{proof}

\begin{remark}
    The presence of $\alpha$ and $\beta$ as coefficients on the $\sfb_{k\rho}$-classes in \cref{thm:HZHomologyProductStructure} is necessary, since the unit in $\Hmr\ull{\zz}_\star$ does not lie in $\Kmr$.
\end{remark}

\subsection{$\Hmr\ull{\zz}$-cohomology of $\Bmr_{\Cmr_2}\Sigma_2$} \label{section: cohomology of KR cohomology of BC2C2} \

To compute the $\Hmr\ull{\zz}$-cohomology ring of $\Bmr_{\Cmr_2}\Sigma_2$, we can apply $F(-, \Hmr\ull{\zz})$ to filtration \eqref{filtration:p(n rho)} to obtain a spectral sequence 
\begin{equation}
\label{sseq:test}
    \Emr_1 = \Hmr\ull{\zz}^\star\langle\sfb^{m\rho}, \sfb^{m\rho + \sigma}\rangle_{m\geq0} \implies \Hmr\ull{\zz}^\star(\Bmr_{\Cmr_2}\Sigma_{2+}).
\end{equation}
By an argument similar to the one for \Cref{thm:HZHomologyBC2Sigma2}, one can show that the only nonzero differentials are determined by ${\rm{d}}_1({\sfb}^{m\rho + \sigma}) = 2\cdot {\sfb}^{(m+1)\rho}$ and that all extension problems are trivial. 

\begin{thm}\label{thm:HZCohomologyBC2Sigma2}
    As $\Hmr\ull{\zz}^\star$-modules,
    \[
        \Hmr\ull{\zz}^\star(\Bmr_{\Cmr_2}\Sigma_{2+}) \cong \Hmr\ull{\zz}^\star\langle \sfb^0\rangle \oplus \Kmr\langle \sfb^{m\rho + \sigma}\rangle_{m\geq 0} \oplus \Cmr\langle\sfb^{m\rho}\rangle_{m\geq 1},
    \]
    where $\Cmr$ and $\Kmr$ are the cokernel and kernel of mutliplication by $2$ on $\Hmr\ull{\zz}^\star$, respectively.
\end{thm}

\begin{remark}
    The $\Cmr$ and $\Kmr$ appearing in \Cref{thm:HZCohomologyBC2Sigma2} are cohomologically graded versions of those appearing in \Cref{thm:HZHomologyBC2Sigma2}.
\end{remark}

The product structure can be deduced by comparison with \eqref{eqn:F2CohomologyBC2Sigma2}.

\begin{corollary}\label{cor:HZCohomologyProductStructure}
    The ring structure of \Cref{thm:HZCohomologyBC2Sigma2} is determined by the following relations:
    \begin{itemize}
        \item $\sfb^{n\rho}\cdot \sfb^{m\rho} = \sfb^{(n+m)\rho}$;
        \item $\sfb^{n\rho}\cdot \sfb^{m\rho +\sigma} = \sfb^{(n+m)\rho + \sigma}$.
    \end{itemize}
\end{corollary}

\begin{remark}
    The classes $\sfb^{n\rho}$ admit a geometric interpretation. In degree $n\rho$, the $\Hmr\ull{\zz}$-cohomology of $\Bmr_{\Cmr_2}\Sigma_2$ is isomorphic to $\ff_2$ by \cref{thm:HZCohomologyBC2Sigma2}.
    Let $\sfe_{\Hmr\ull{\zz}}$ denote the $\Hmr\ull{\zz}$-Euler class of $\xi_\rho$ associated to the orienation we fixed at the start of this section. The ring map $\Hmr\ull{\zz} \to \Hmr\ull{\ff}_2$ sends $\sfe_{\Hmr\ull{\zz}}^i$ to $\sfe_{\Hmr\ull{\ff}_2}^i$ for all $i\geq 1$. Consequently $\sfe_{\Hmr\ull{\zz}}^i$ is nonzero and therefore equal to $\sfb^{i\rho}$.
\end{remark}


\section{The Real K-theory of $\text{C}_2$-equivariant classifying space of $\Sigma_2$}
\label{section 3}

Classically, all nontrivial differentials in \eqref{eqn:kuSliceSS}, which computes the $\text{ku}$-homology of $\text{B}\Sigma_2$, can be determined by comparison with the spectral sequence
\begin{equation}
    \Emr_1 = \bigoplus_{m\geq0}\Sigma^{2m}\Hmr\zz_*(S^0) \implies \text{ku}_*(S^0)
\end{equation}
via the map $\Bmr\Sigma_{2+}\to S^0$.  Equivariantly, we study the $\rm{k}\r$-homology of $\Bmr_{\Cmr_2}\Sigma_2$ in a similar manner, by comparing \eqref{eqn:kRSS} with the slice spectral sequence for $\text{k}\mathbb{R}$.

\subsection{Augmented slice spectral sequence} \

Our main technique is a variation of the localized spectral sequence introduced by Meier--Shi--Zeng \cite[Theorem 3.3]{MSZ23}. The construction and convergence follow \emph{mutatis mutandis} from their work. For completeness, we include a brief proof illustrating the strategy. 



\begin{thm} \label{thm: main computational technique for KR}
    Fix a finite group $\Gmr$. Let $\{ \text{P}^{\bullet}\Emr \}$ denote the regular slice tower (see \cite{Ull13}) for a genuine $\Gmr$-equivariant spectrum $\Emr$ and $X$ be a based $\mathrm{G}$-space. The spectral sequence associated to the filtration 
    \[
        \{ \text{Q}^{\bullet} \} \coloneqq \{ X \wedge \text{P}^{\bullet}\Emr \}
    \]
    has signature
    \begin{equation} \label{signature of ASSS}
        \Emr^{s, \Vmr}_2 = \pi_{\Vmr-s}^{\Gmr}(\text{Q}_{|\Vmr|}^{|\Vmr|}),
    \end{equation}
    and converges strongly to $\pi_{\Vmr-s}^{\Gmr}(\Emr \wedge X)$. 
\end{thm}

\begin{proof}
    In the tower $\{ \text{Q}^{\bullet} \}$, we note that for every subgroup $H \subseteq G$,
    \begin{align*}
        \Phi^{H}(\colim_n(X \wedge \text{P}^n \Emr)) & \simeq \Phi^{H}(X) \wedge \colim_n \Phi^H(\text{P}^n \Emr) \\
            & \simeq \Phi^{H}(X) \wedge * \simeq *.
    \end{align*}
    It follows that the colimit of the tower is contractible. Since $X \geq 0$ and $\mathrm{P}_n^n \Emr \geq n$, the space $X$ plays the same role as $S^{\infty \lambda}$ in the proof of \cite[Theorem 3.3]{MSZ23}. The proposition then follows by the same argument, replacing $S^{\infty \lambda}$ with $X$.
\end{proof}

\begin{remark}
    In \cref{thm: main computational technique for KR}, we focus on the regular slice filtration introduced in \cite{Ull13}. Unlike the slice filtration of \cite{HHR}, the regular slice filtration involves only regular slice cells of the form $\Gmr_+ \wedge S^{k\rho_H}$ for $H \subseteq \Gmr$. The relationship between these two filtrations is discussed in \cite[Proposition I.3.4]{Ull13}. In the case of interest, such as $\text{E} = \text{k}\r$, the two filtrations agree. Thus, we will use the terms ``slice filtration” and ``regular slice filtration” interchangeably for the remainder of the paper.
\end{remark}

\begin{definition}
    Fix a finite group $\Gmr$. For any $\text{G}$-space $Y$, we call the spectral sequence arising from \Cref{thm: main computational technique for KR}, which computes $\pi_\star^\text{G}(\Emr \wedge Y_+)$, the \emph{augmented slice spectral sequence}, denoted $\textbf{ASliceSS}_\Emr(Y)$.
\end{definition}

If $\Ymr$ is a $\Gmr$-space, then there is an augmentation
\[\Emr \wedge \Ymr_+ \to \Emr \wedge \Smr^0, \]
which sends $Y$ to the non-basepoint of $\Smr^0$.
This induces a map of spectral sequences
\begin{equation}\label{eqn:augmentationMap}
    \varepsilon:\textbf{ASliceSS}_\Emr(\Ymr) \to \textbf{SliceSS}(\Emr)
\end{equation}
which we will refer to as an \emph{augmentation}.
\begin{remark} 
    When $X = \widetilde{\Emr}\Cmr_2$, the spectral sequence associated to $\text{Q}^\bullet$ is the $a_\sigma$-localized slice spectral sequence considered in \cite{MSZ23} as $\widetilde{\Emr}\Cmr_2\wedge \mathrm{E}\simeq a_\sigma^{-1} \mathrm{E}$. 
\end{remark}
\begin{remark} \label{rem: lax monidality}
    If $Y$ is a $\Cmr_2$-equivariant H-space and the (regular) slice filtration for $\Emr$ is multiplicative, then $\textbf{ASliceSS}_\Emr(Y)$ is multiplicative. This follows from the discussion in \cite[Chapter I.4]{Ull13}. In this case, the functor $\text{Q}^{\bullet}(-)$ is lax monoidal since $\text{P}^{\bullet}(-)$ is.
\end{remark}
When $\Gmr = \Cmr_{2}$, applying \Cref{thm: main computational technique for KR} with $\Xmr = \widetilde{\Emr}\Cmr_2\wedge \Ymr_+$ leads to the construction of a {\it localized augmented slice spectral sequence}
\begin{equation}\label{eqn:locSliceSS}
    \Emr_2^{s,\Vmr} = \pi_{\Vmr - s}^{\text{C}_2}((\text{P}^\bullet\Emr)^\Phi \wedge Y_+) \implies \pi_{\Vmr-s}^{\text{C}_2}(\Emr^\Phi\wedge Y_+),
\end{equation} 
where $(-)^\Phi \coloneqq \widetilde{\Emr}\Cmr_2\wedge -$ is being applied to each term in the filtration $\text{P}^\bullet \Emr$.
Notice that the adjectives ``localized" and ``augmented" commute in the description of this spectral sequence. In particular, \eqref{eqn:locSliceSS} is identical to $a_\sigma^{-1}\textbf{ASliceSS}_{\Emr}(Y) = \textbf{ASliceSS}_{\Emr^\Phi}(Y)$. The ring map $\varphi:\Emr\to \Emr^\Phi$ induces a diagram of spectral sequences
\begin{equation}\label{eqn:localizedAndAugmentedSliceSS}
    \begin{tikzcd}
        \textbf{ASliceSS}_{\Emr}(Y) \ar[r, "\varphi"] \ar[d, "\varepsilon"'] & \textbf{ASliceSS}_{\Emr^\Phi}(Y) \ar[d, "\varepsilon"]\\
        \textbf{SliceSS}(\Emr) \ar[r, "\varphi"] & \textbf{SliceSS}(\Emr^\Phi)
    \end{tikzcd}
\end{equation}
We will deduce differentials in $\textbf{ASliceSS}_{\rm{k}\r}(\Bmr_{\Cmr_2}\Sigma_{2+})$ via its augmentation and by comparing with $\textbf{ASliceSS}_{\rm{k}\r^\Phi}(\Bmr_{\Cmr_2}\Sigma_{2})$.


\subsection{$\text{k}{\r}$-homology of $\Bmr_{\Cmr_2} \Sigma_2$} \

By \eqref{eqn:kRSlice} and \Cref{thm:HZHomologyBC2Sigma2}, we can describe the $\Emr_2$-page of $\textbf{ASliceSS}_{\text{k}\r}(\Bmr_{\Cmr_2}\Sigma_{2})$ as a graded ring:
\begin{equation}\label{eqn:E2PageASSS}
    \Emr_2(\textbf{ASliceSS}_{\rm{k}\r}(\Bmr_{\Cmr_2}\Sigma_{2})) \cong \Hmr\ull{\zz}_\star(\Bmr_{\Cmr_2}\Sigma_{2+})[\overline{v_1}],
\end{equation}
where $|\overline{v_1}| = (2, \rho)$. 

As the structure of $\text{k}\r$ will soon be crucial to our arguments, we include a detailed chart describing the coefficients $\text{k}\r_\star$ in \Cref{fig:kR}.
We use the following notation \cite{Greenlees17}:
\[
\text{k}\mathbb{R}_\star = \frac{\mathbb{Z}[2u_{2\sigma}, u_{4\sigma,} a_\sigma, \overline{v_1}]}{(2a_\sigma, a_\sigma^3\overline{v_1}, (2u_{2\sigma})^2 = u_{4\sigma})} \oplus \bigoplus_{i \geq 0}\mathbb{Z}\left\{\frac{2\theta_2}{2u_{2\sigma}^i} \right\}[\overline{v_1}] \oplus \bigoplus_{\substack{j \geq 0 \\ k \geq 0}}\mathbb{F}_2\left\{\frac{\theta_5}{u_{4\sigma}^j a_{\sigma}^k}\right\} \oplus \bigoplus_{\ell \geq 1} \mathbb{Z}\left\{\frac{\theta_2\overline{v_1}}{u_{4\sigma}^{\ell}} \right\}[\overline{v_1}]
\]

In these charts, a $\blacksquare$ denotes a $\zz$, and a \textcolor{blue}{blue $\bullet$} denotes a $\zz/2$. A \textcolor{blue}{blue} vertical line represents $a_\sigma$-multiplication, and a \textcolor{purple}{purple} line of slope 1 represents $\overline{v_1}$-multiplication.
\vfill
\begin{figure}[H]
    \centering
    \includegraphics[scale=.7]{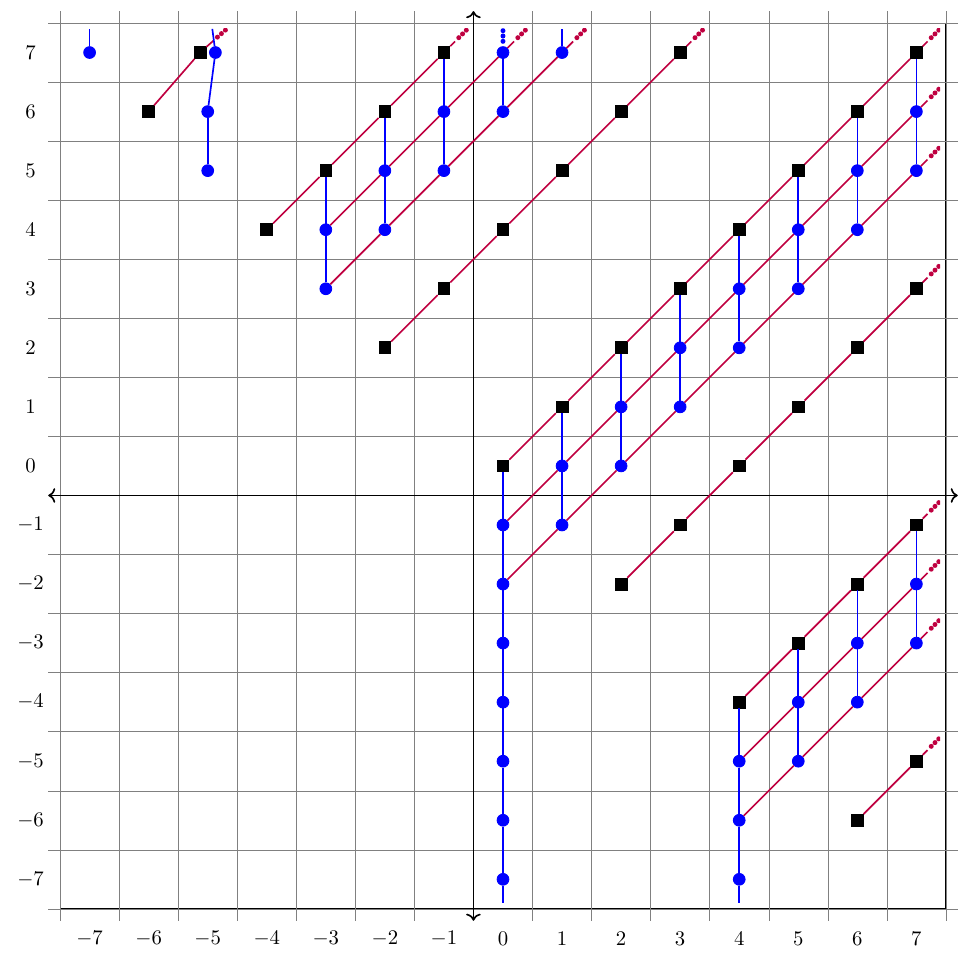}
    \caption{The coefficients $\text{k}\r$.}
    \label{fig:kR}
\end{figure}

The only nontrivial differentials in $\textbf{SliceSS}(\rm{k}\r)$ are generated by
\begin{equation}\label{eqn:differentialSSSkR}
    \text{d}_3(u_{2\sigma}) = a_\sigma^3 \overline{v_1}
\end{equation}
and the Leibniz rule \cite{Dugger}. By pulling back along the augmentation \eqref{eqn:augmentationMap}, this implies that the same differential occurs on the $\Hmr\ull{\z}_\star$-summand of \eqref{eqn:E2PageASSS}.

We depict the $\text{d}_3$-differential on the $\Hmr\ull{\zz}_\star$-summand in \Cref{fig:d3_on_HZ}. In this chart, we have suppressed filtration degrees.

\begin{figure}[H]
    \centering
    \includegraphics[scale=.7]{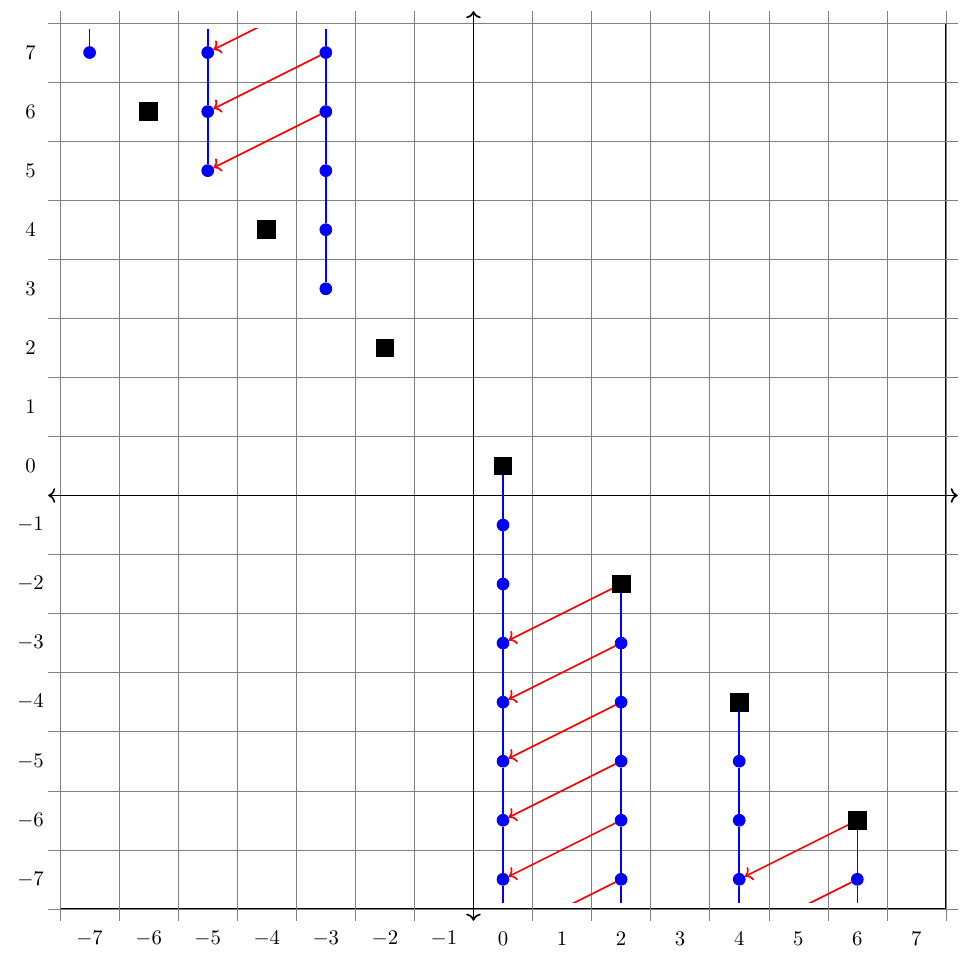}
    \caption{The $\text{d}_3$-differential of \eqref{eqn:E2PageASSS} on the $\Hmr\ull{\z}_\star$ summand.}
    \label{fig:d3_on_HZ}
\end{figure}

The $\Emr_4$-page takes two different forms. In filtration 0, the entire kernel of $\text{d}_3:\Hmr\ull{\z}_\star \to \Hmr\ull{\z}_\star[\overline{v_1}]$ survives, since there is no incoming differential. In higher filtrations, however, there is an incoming differential, resulting in a nontrivial quotient. For the reader’s convenience, we present charts for both cases in \Cref{fig:ker_d3_on_HZ} and \Cref{fig:E4_HZ}, respectively.

\begin{figure}
    \begin{minipage}{.45\textwidth}
    \centering
    \includegraphics[scale=.4]{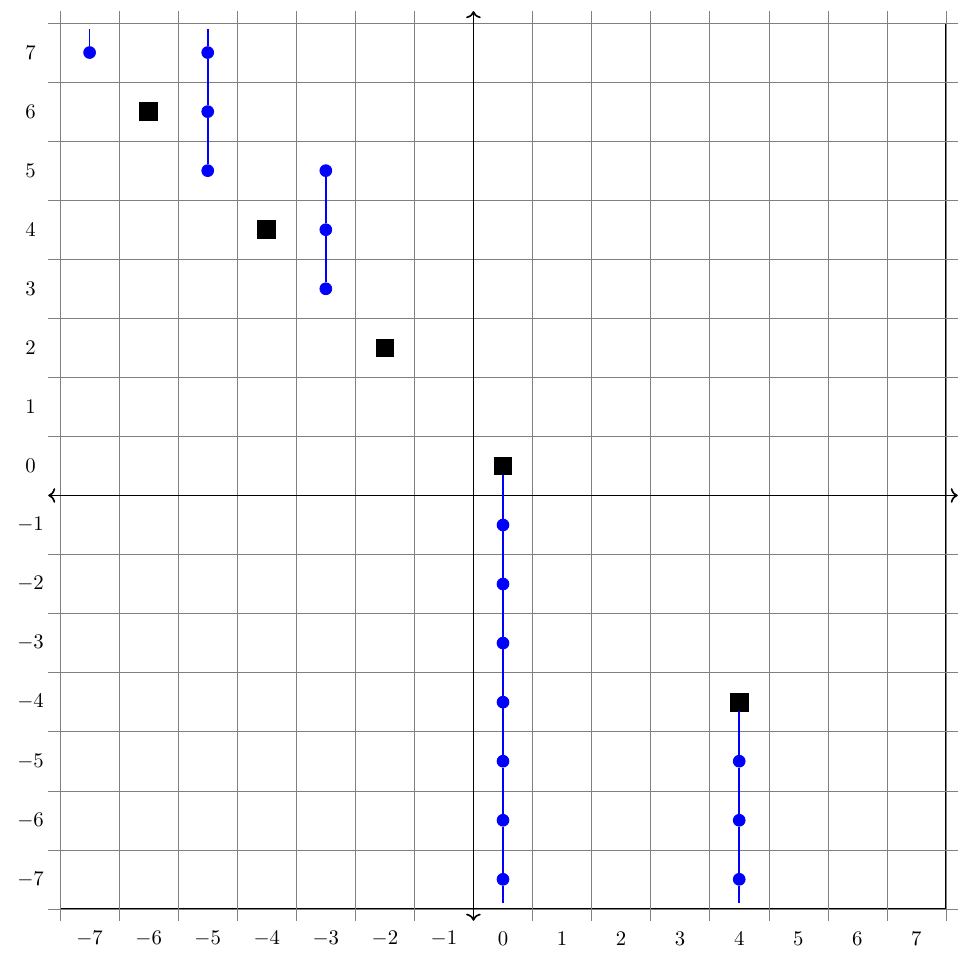}
    \caption{The $\Hmr\ull{\z}_\star$-module \\ $ \text{ker}(\Hmr\ull{\z}_\star \xrightarrow{\text{d}_3}\Hmr\ull{\z}_\star\langle\overline{v_1}\rangle)$.}
    \label{fig:ker_d3_on_HZ}
    \end{minipage}
    \begin{minipage}{.45\textwidth}
        \includegraphics[scale=.4]{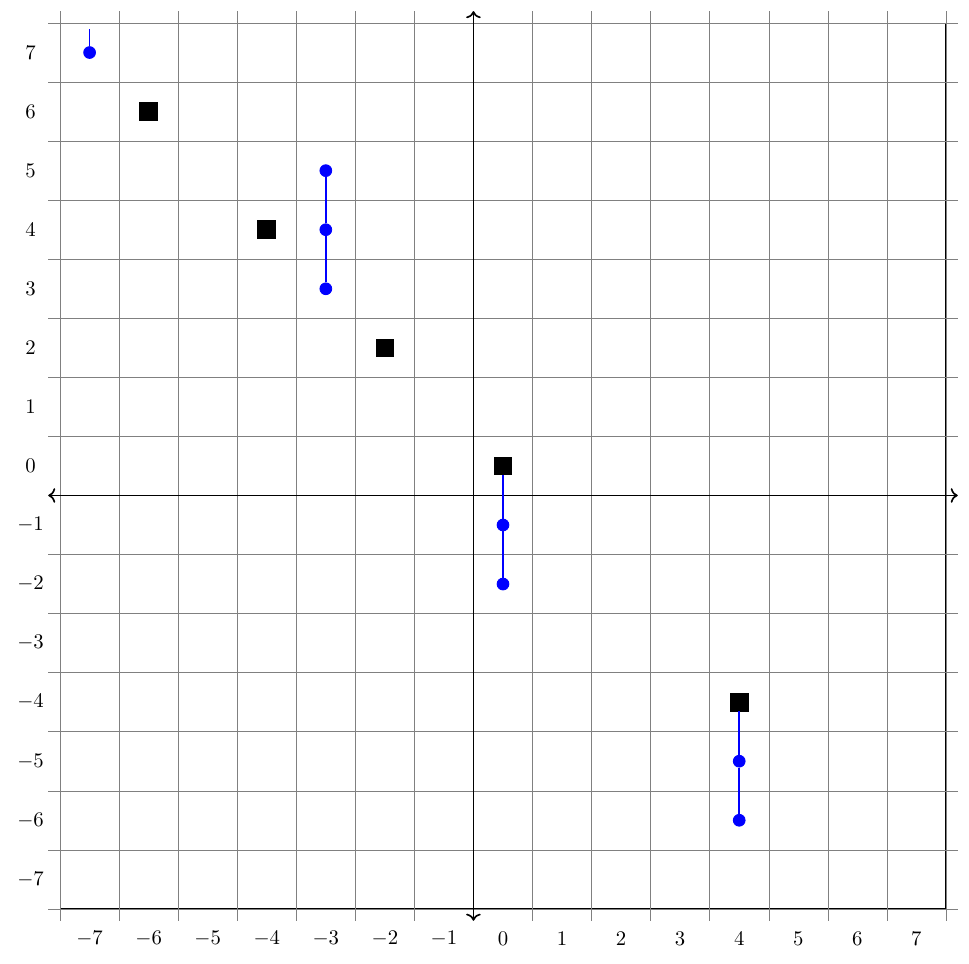}
        \caption{The $\Hmr\ull{\z}_\star$-module $\Hmr_*(\Hmr\ull{\z}_\star\langle \overline{v_1}^m \rangle; \text{d}_3)$.}
        \label{fig:E4_HZ}
    \end{minipage}
\end{figure}

\begin{remark}
One may also interpret the above charts as depicting the $\text{d}_3$-differential in the $\textbf{SliceSS}(\text{k}\r)$.
The structure of the $\Emr_\infty$-page is immediate from \Cref{fig:ker_d3_on_HZ} and \Cref{fig:E4_HZ}, since the nontrivial slices of $\text{k}\r$ are concentrated in degrees that are multiples of the regular representation, and there are no hidden extensions on the $\Emr_\infty$-page. 
\end{remark}



\begin{lemma}\label{lem:ASSSkRCDiffs}
    For each $n\geq 0$, the only nonzero $\text{d}_3$-differentials on the $\Cmr$-summands of \eqref{eqn:E2PageASSS} are determined by
    \begin{equation}
        \text{d}_3(u_{2\sigma}\sfb_{n\rho+\sigma}) = a_\sigma^3\sfb_{n\rho+\sigma}\overline{v_1}
    \end{equation}
\end{lemma}

\begin{proof}
    The differential \eqref{eqn:differentialSSSkR} lifts via the augmentation to a differential in $\textbf{ASliceSS}_{{\rm k}\r}(\Bmr_{\Cmr_2}\Sigma_{2})$ given by
    \[
        \text{d}_3(u_{2\sigma}\sfb_0) = a_\sigma^3 \overline{v_1} \sfb_0.
    \]
    Since $\sfb_0$ is the monoidal unit in the multiplicative structure on the $\Hmr\ull{\zz}$-homology of $\Bmr_{\Cmr_2}\Sigma_{2+}$, we can deduce
    \begin{align*}
        \begin{split}
            \text{d}_3(u_{2\sigma}\sfb_{n\rho+\sigma}) &= \text{d}_3(u_{2\sigma} \sfb_0 \cdot \sfb_{n\rho+\sigma})\\
            &= \text{d}_3(u_{2\sigma}\sfb_0)\sfb_{n\rho+\sigma} + u_{2\sigma}\sfb_0 \cdot \text{d}_3(\sfb_{n\rho+\sigma})\\
                &= a_\sigma^3\overline{v_1} \sfb_{n\rho+\sigma} \,,
        \end{split}
    \end{align*}
    where we have used that $\text{d}_3(\sfb_{n\rho+\sigma}) = 0$ for degree reasons. Extending via the Leibniz rule combined the with $\Hmr\ull{\zz}_\star$-module structure now completely determines the behavior of the $\text{d}_3$-differential on any $\text{C}$-summand.
\end{proof}

We depict the $\text{d}_3$-differential from \Cref{lem:ASSSkRCDiffs} in \Cref{fig:d3_on_C}. In this chart, we have suppressed filtration degrees. Just as in the case of the $\Hmr\ull{\z}_\star$-summand, the $\Emr_4$-page takes two different forms. For the reader’s convenience, we present charts for both cases in \Cref{fig:ker_d3_on_C} and \Cref{fig:E4_C}, respectively.

\begin{figure}[H]
    \centering
    \includegraphics[scale=.7]{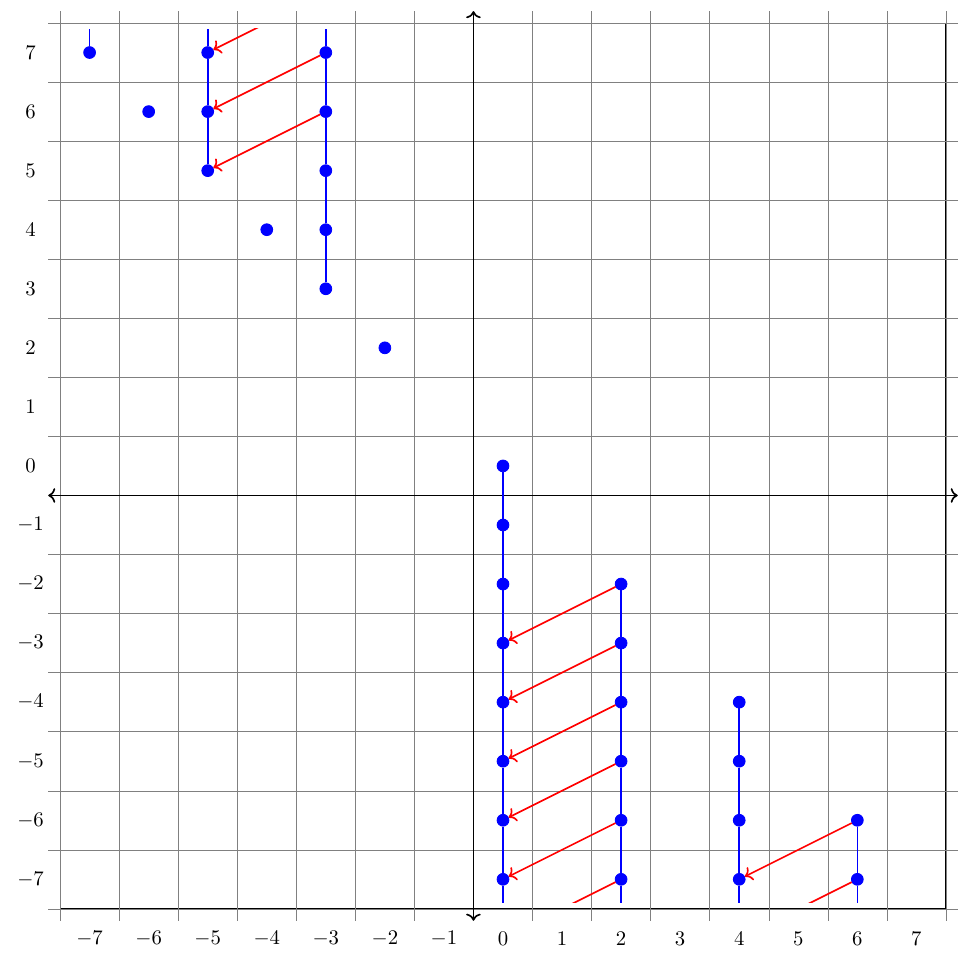}
    \caption{The $\text{d}_3$-differential of \eqref{eqn:E2PageASSS} on the C summand.}
    \label{fig:d3_on_C}
\end{figure}

\begin{figure}[H]
    \begin{minipage}{.45\textwidth}
    \centering
    \includegraphics[scale=.4]{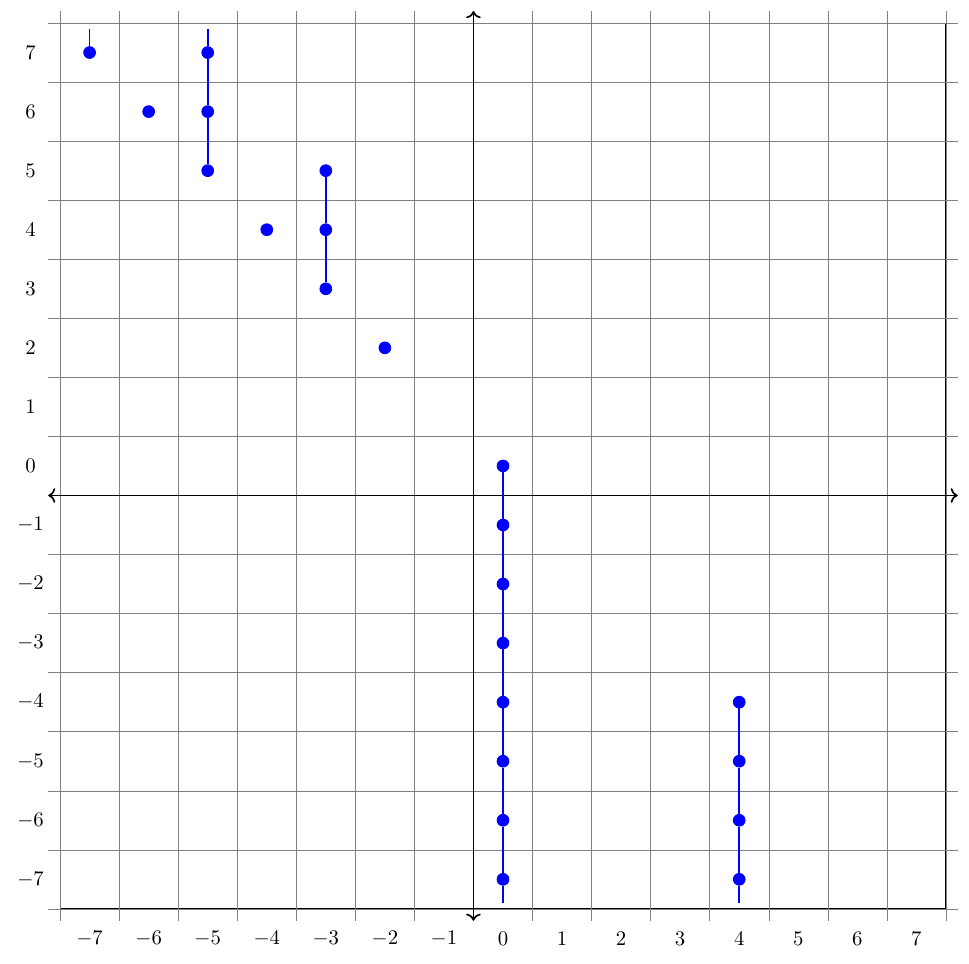}
    \caption{The group $ \text{ker}(\text{C} \xrightarrow{\text{d}_3}\text{C}\langle\overline{v_1}\rangle)$.}
    \label{fig:ker_d3_on_C}
    \end{minipage}
    \begin{minipage}{.45\textwidth}
        \includegraphics[scale=.4]{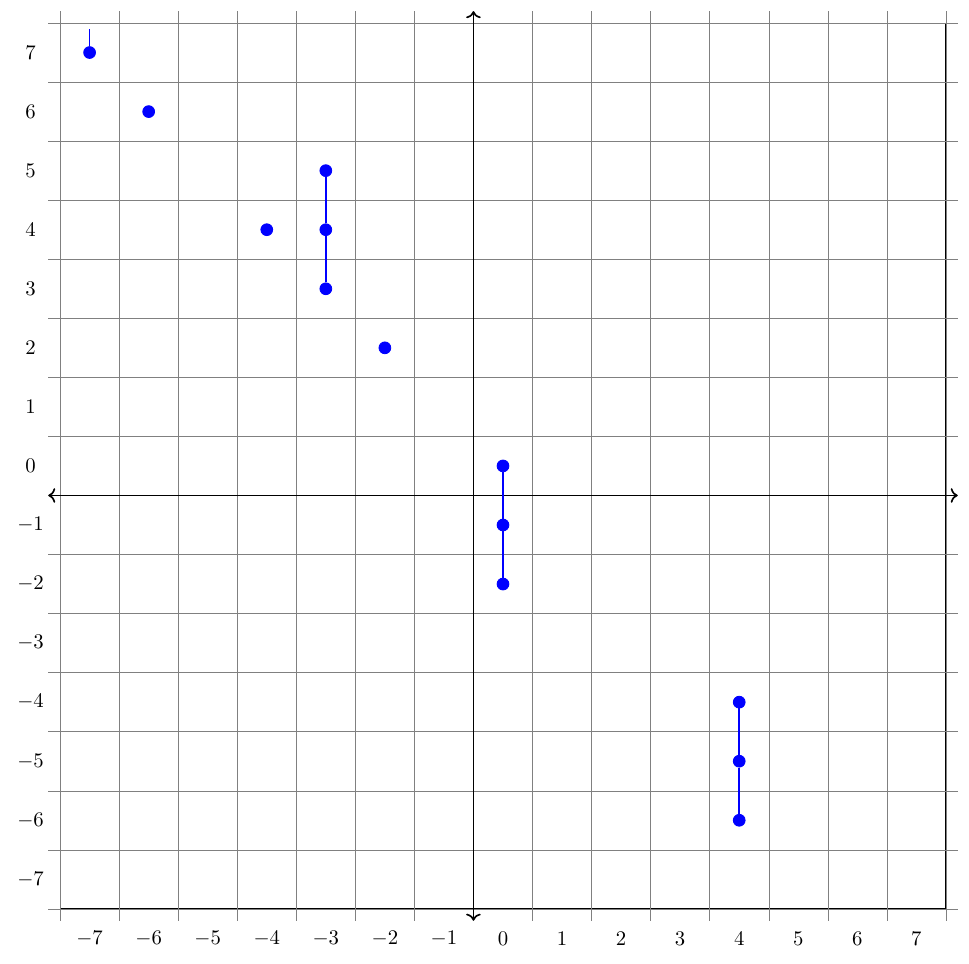}
        \caption{The group \\ $\Hmr_*(\text{C}\langle \overline{v_1}^{m} \rangle; \text{d}_3\rangle$.}
        \label{fig:E4_C}
    \end{minipage}
\end{figure}


To study $\text{d}_3$-differentials on the $\Kmr$-summand, we begin by observing that if $\text{d}_3(a_\sigma\sfb_{n\rho})$ is nonzero, then for degree reasons it is necessarily equal to $a_\sigma \sfb_{(n-1)\rho+\sigma} \overline{v_1}$. Since both the source and target of this potential differential are $a_\sigma$ torsion-free, this differential will be witnessed in $\textbf{ASliceSS}_{\text{k}\r^\Phi}(\Bmr_{\Cmr_2}\Sigma_{2})$. The following lemma is immediate.
\begin{lemma}
    There are isomorphisms of graded rings
    \begin{align*}
    \Hmr\ull{\zz}_\star^{\Phi} &\cong \ff_2[a_\sigma^{\pm}][u_{2\sigma}], \\
    \text{k}\r_\star^\Phi &\cong \ff_2[a_\sigma^{\pm}][u_{4\sigma}].
    \end{align*}
\end{lemma}


It follows that
\[
    \Emr_2(\textbf{SliceSS}(\text{k}\r^\Phi)) \cong \ff_2[a_\sigma^{\pm}][u_{2\sigma}][\overline{v_1}],
\]
therefore all nontrivial differentials in this spectral sequence are determined by
\[
    \text{d}_3(u_{2\sigma}\overline{v_1}^i) = a_\sigma^3\overline{v_1}^{i+1}. 
\]
Since the $\Emr_\infty$-page of this spectral sequence is concentrated in filtration $0$, there is no room for extensions. 

\begin{proposition}
    There is an isomorphism of graded rings
    \[
        \Hmr\ull{\zz}^\Phi_\star(\Bmr_{\Cmr_2}\Sigma_{2+})\cong \Hmr\ull{\zz}^\Phi_\star\langle \sfb_{n\rho}, \sfb_{n\rho+\sigma}\rangle_{n\geq 0}
    \]
    where
    \begin{itemize}
        \item $\sfb_{n\rho}\cdot \sfb_{m\rho} = (1-\delta_n^m)\sfb_{(n+m)\rho}$ when $n,m\geq 1$;
        \item $\sfb_{n\rho+\sigma}\cdot \sfb_{m\rho + \sigma} = 0$;
        \item $\sfb_{n\rho}\cdot \sfb_{m\rho + \sigma} = \sfb_{(n+m)\rho+\sigma}$.
    \end{itemize}
\end{proposition}

\begin{proof}
    The $\Hmr\ull{\zz}^\Phi$-homology spectral sequence associated to \eqref{filtration:p(n rho)} collapses on the $\Emr_1$-page since $\Hmr\ull{\zz}^\Phi_0\cong \ff_2$ and there are no higher differentials in the analogous $\Hmr\ull{\ff}_2^\Phi$-homology spectral sequence. There are no extension problems as the $\Emr_1$-page is free. The product structure follows by comparing with the $\Hmr\ull{\ff}_2^\Phi$-homology spectral sequence, whose product structure is completely determined by \Cref{cor:HF2HomologyProductStructure}.  
\end{proof}

Since $\text{k}\r^\Phi_\star\xrightarrow[]{\varphi_*} \Hmr\ull{\zz}^\Phi_\star$ is an inclusion of a subalgebra, it follows that the $\text{k}\r^{\Phi}$-homology spectral sequence associated to \eqref{filtration:p(n rho)} also collapses immediately. 

\begin{corollary}
\label{cor:krPhiBC2Sigma2}
    There is an isomorphism of graded rings,
    \[
        \text{k}\r^\Phi_\star(\Bmr_{\Cmr_2}\Sigma_{2+})\cong \text{k}\r^\Phi_\star\langle \sfb_{n\rho}, \sfb_{n\rho+\sigma}\rangle_{n\geq 0}
    \]
    where
    \begin{itemize}
        \item $\sfb_{n\rho}\cdot \sfb_{m\rho} = (1-\delta_n^m)\sfb_{(n+m)\rho}$ when $n,m\geq 1$;
        \item $\sfb_{n\rho+\sigma}\cdot \sfb_{m\rho + \sigma} = 0$;
        \item $\sfb_{n\rho}\cdot \sfb_{m\rho + \sigma} = \sfb_{(n+m)\rho+\sigma}$.
    \end{itemize}
\end{corollary}

We are now able to deduce all of the differentials in $\textbf{ASliceSS}_{\text{k}\r^\Phi}(\Bmr_{\Cmr_2}\Sigma_2)$.

\begin{proposition}\label{prop:kRPhiBC2Sigma2SS}
    The nontrivial differentials in $\textbf{ASliceSS}_{\text{k}\r^\Phi}(\Bmr_{\Cmr_2}\Sigma_2)$ are determined by
    \begin{align*}
        \begin{split}
            \text{d}_3(u_{2\sigma}\sfb_{n\rho}) &= a_\sigma^3 \sfb_{n\rho}\overline{v_1},\\
            \text{d}_3(u_{2\sigma}\sfb_{n\rho+\sigma}) &= a_\sigma^3\sfb_{n\rho+\sigma}\overline{v_1}.
        \end{split}
    \end{align*}
\end{proposition}

\begin{proof}
    Additively, $\textbf{ASliceSS}_{\text{k}\r^\Phi}(\Bmr_{\Cmr_2}\Sigma_2)$ is a direct sum of copies of the $\textbf{SliceSS}(\text{k}\r^\Phi)$-spectral sequence, with summands indexed by $\sfb_{n\rho}$ and $\sfb_{n\rho + \sigma}$. By \Cref{cor:krPhiBC2Sigma2}, the $\text{d}_3$-differentials listed are the only possible for degree reasons. Since $a_\sigma$ is a permanent cycle, this completely determines the $\text{d}_3$-differential. For degree reasons, the spectral sequence collapses at $\Emr_4$.
\end{proof}


The differentials on the $\Kmr$-summands in $\Emr_3(\textbf{ASliceSS}_{\text{k}\r}(\Bmr_{\Cmr_2}\Sigma_2))$ can now be deduced by comparison with $\textbf{ASliceSS}_{\text{k}\r^\Phi}(\Bmr_{\Cmr_2}\Sigma_2)$.

\begin{lemma}\label{lem:ASSSkRKDiffs}
    For each $n\geq 0$, the $\text{d}_3$-differentials on the $\Kmr$-factors of \eqref{eqn:E2PageASSS} are determined by
    \begin{equation}
        \text{d}_3(a_\sigma u_{2\sigma}\sfb_{n\rho}) = a_\sigma^4\overline{v_1} \sfb_{n\rho}.
    \end{equation}
\end{lemma}

\begin{proof}
    Recall that, for degree reasons, if $\text{d}_3(a_\sigma \sfb_{n\rho})$ is nonzero, then it is equal to $a_\sigma\sfb_{(n-1)\rho+\sigma}\overline{v_1}$. The source and target of this differential are witnessed in $\textbf{ASliceSS}_{\text{k}\r^\Phi}(\Bmr_{\Cmr_2}\Sigma_2)$, where this differential is known not to occur due to \Cref{prop:kRPhiBC2Sigma2SS}. So 
    \[\text{d}_3(a_\sigma \sfb_{n\rho}) = 0,\qquad n\geq 0.\]
    Recall that the augmentation for $\textbf{ASliceSS}_{\text{k}\r}(\Bmr_{\Cmr_2}\Sigma_2)$ yields a differential
    \[
        \text{d}_3(u_{2\sigma}\sfb_0) = a_\sigma^3 \overline{v_1} \sfb_0 .
    \]
    From the Leibniz rule and the multiplicative structure in \Cref{thm:HZHomologyProductStructure}, 
    \begin{align*}
        \text{d}_3(a_\sigma u_{2\sigma} \sfb_{n\rho}) &= \text{d}_3(a_\sigma u_{2\sigma} \sfb_0 \cdot \sfb_{n\rho}) \\
          &=\text{d}_3(u_{2\sigma}\sfb_0) \cdot a_\sigma\sfb_{n\rho} + u_{2\sigma}\sfb_0 \cdot \text{d}_3(a_\sigma \sfb_{n\rho}) \\
          &= a_\sigma^4 \overline{v_1} \sfb_{n\rho}.
    \end{align*}
    These two differentials determine the behavior of $\text{d}_3$ on $a_\sigma \ff_2[a_\sigma, u_{2\sigma}]\subset \Kmr$. An arbitrary element in the complement of this subalgebra has the form $\theta_2/a_\sigma^iu_{2\sigma}^j$ where $i,j\geq 0$. For degree reasons, the differentials on $\theta_2\sfb_{n\rho}$ are trivial. The remaining differentials are determined by the Leibniz rule.
\end{proof}

We depict the $\text{d}_3$-differential from \Cref{lem:ASSSkRKDiffs} in \Cref{fig:d3_on_K}. In this chart, we have suppressed filtration degrees. As was the case for the $\Hmr\ull{\z}_\star$-summand, the $\Emr_4$-page takes two different forms. For the reader’s convenience, we present charts for both cases in \Cref{fig:ker_d3_on_K} and \Cref{fig:E4_K}, respectively.

\begin{figure}[H]
    \centering
    \includegraphics[scale=.7]{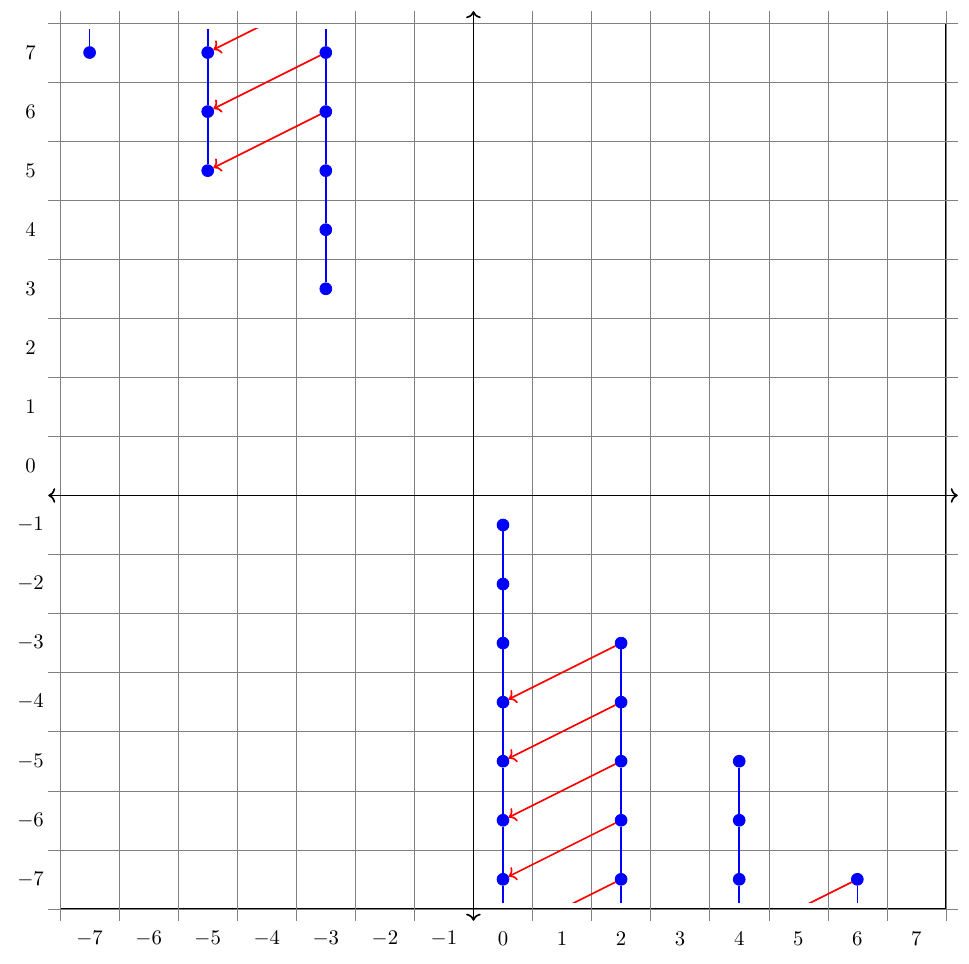}
    \caption{The $\text{d}_3$-differential of \eqref{eqn:E2PageASSS} on the K summand.}
    \label{fig:d3_on_K}
\end{figure}

\begin{figure}[H]
    \begin{minipage}{.45\textwidth}
    \centering
    \includegraphics[scale=.4]{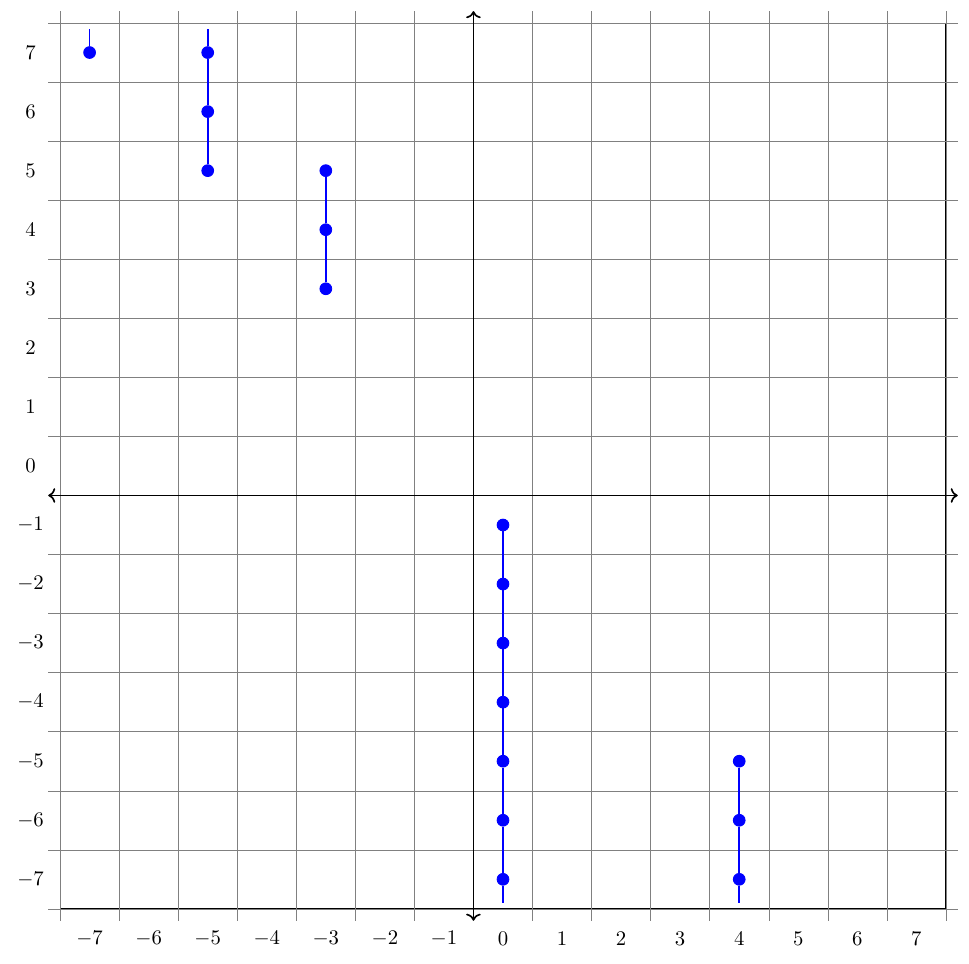}
    \caption{The module $ \text{ker}(\text{K} \xrightarrow{\text{d}_3}\text{K}\langle\overline{v_1}\rangle)$.}
    \label{fig:ker_d3_on_K}
    \end{minipage}
    \begin{minipage}{.45\textwidth}
        \includegraphics[scale=.4]{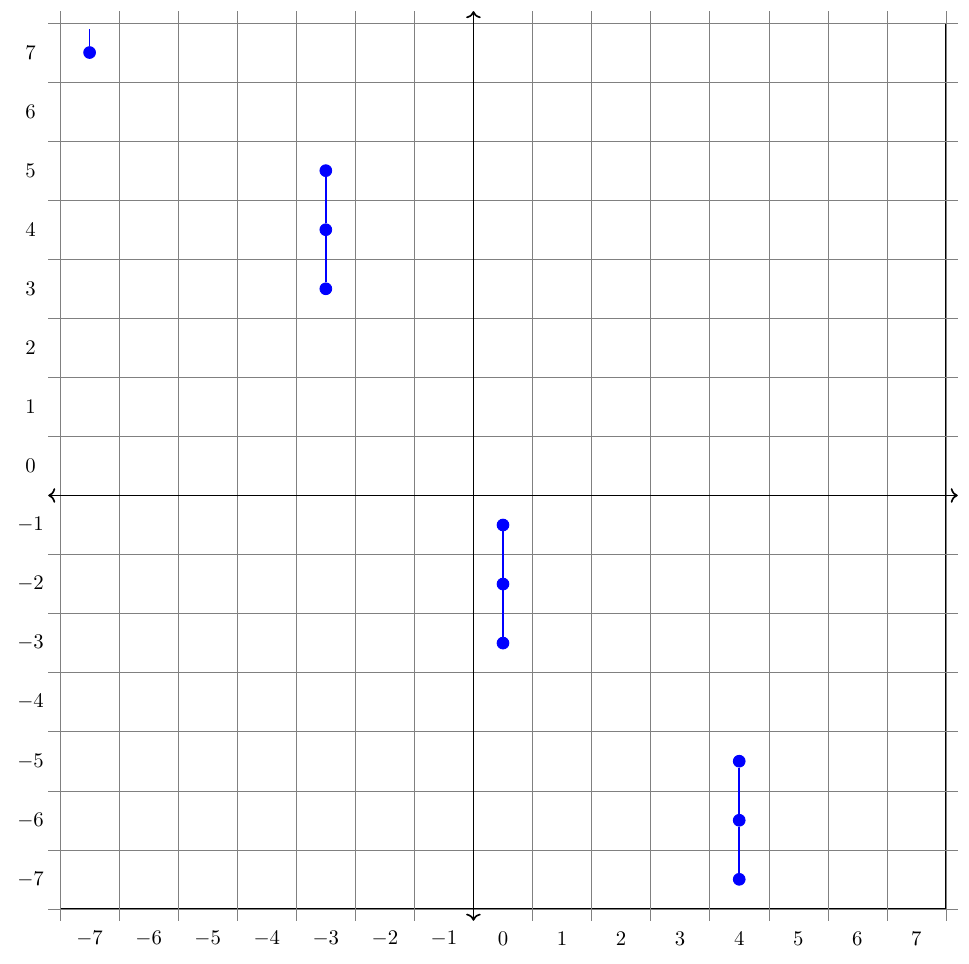}
        \caption{The module \\ $\Hmr_*(\text{K}\langle \overline{v_1}^m \rangle; \text{d}_3)$}
        \label{fig:E4_K}
    \end{minipage}
\end{figure}


Together with the $\text{d}_3$-differential on the $\Hmr\ull{\z}_\star$-summand, \Cref{lem:ASSSkRCDiffs} and \Cref{lem:ASSSkRKDiffs} show that the $\text{d}_3$-differentials take the form
\begin{align*}
    \begin{split}
     \text{d}_3&:\Hmr\ull{\zz}_\star[\overline{v_1}] \to \Hmr\ull{\zz}_\star[\overline{v_1}],\\
        \text{d}_3&:\Cmr[\overline{v_1}]\to \Cmr[\overline{v_1}],\\
        \text{d}_3&:\Kmr[\overline{v_1}]\to \Kmr[\overline{v_1}].
    \end{split}
\end{align*} 
We introduce the following notation to describe $\Emr_4(\textbf{ASliceSS}_{\text{k}\r}(\Bmr_{\Cmr_2}\Sigma_{2}))$.

\begin{definition}\label{defn:assocGradedPieceskRSS}
    Let $\overline{\text{D}}$, $\overline{\Cmr}$, and $\overline{\Kmr}$ denote the homology of $\Hmr\ull{\zz}_\star[\overline{v_1}]$, $\Cmr[\overline{v_1}]$ and $\Kmr[\overline{v_1}]$ under the the $\text{d}_3$-differential, respectively. Namely,
    \begin{align*}
        & \overline{\text{D}} = \Hmr_*(\Hmr\ull{\zz}_\star[\overline{v_1}], \text{d}_3), \\
        & \overline{\Cmr} = \Hmr_*(\Cmr[\overline{v_1}], \text{d}_3), \\
        & \overline{\Kmr} = \Hmr_*(\Kmr[\overline{v_1}], \text{d}_3).
    \end{align*}
\end{definition}

\begin{remark}
    The algebra $\overline{\Dmr}$ is the usual $\Emr_\infty$-page of $\textbf{SliceSS}(\text{k}\r)$.
\end{remark}

\begin{corollary}\label{cor:E3PagekRASSS}
    The $\Emr_4$-page of $\textbf{ASliceSS}_{\text{k}\r}(\Bmr_{\Cmr_2}\Sigma_{2})$ is
    \[
        \overline{\Dmr}\langle \sfb_0\rangle \oplus \overline{\Cmr}\langle \sfb_{n\rho+\sigma}\rangle_{n\geq 0}\oplus \overline{\Kmr}\langle \sfb_{n\rho}\rangle_{n\geq 1}.
    \]
\end{corollary}

\begin{thm}\label{thm:kRAssociatedGraded}
    The spectral sequence $\textbf{ASliceSS}_{\text{k}\r}(\Bmr_{\Cmr_2}\Sigma_{2})$ collapses on the $\Emr_4$-page. 
\end{thm}

\begin{proof}

    For degree reasons, and by comparison via the augmentation \eqref{eqn:augmentationMap}, there are no further differentials on $\overline{\text{D}}$ or $\overline{\text{C}}$. Again, degree reasons, together with the fact that $\overline{v_1}a_\sigma^3=0$, show that for $r \geq 4$, we have
    \begin{equation}
    \label{eq:dr_asigma_brho_0}
    \text{d}_r(a_\sigma{\sf{b}}_\rho) = \text{d}_r(a_\sigma {\sf{b}}_{2\rho}) = 0.
    \end{equation}
    For all $n \geq 3$ and $\ell\geq 0$, there is a potential differential
    \[
    \text{d}_{2n-1}(a_\sigma{\sf{b}}_{(2n-3 + \ell)\rho}) = a_\sigma{\sf{b}}_{\ell\rho + \sigma}\overline{v_1}^{2n-5}.
    \]
    Observe that
    \[
    \text{d}_{2n-1}(a_\sigma^2{\sf{b}}_{(2n-3 + \ell)\rho}) = a_\sigma{\sf{b}}_{(2n-4 + \ell)\rho}\text{d}_{2n-1}(a_\sigma{\sf{b}}_{\rho}) + \text{d}_{2n-1}(a_\sigma{\sf{b}}_{(2n-4 + \ell)\rho}).
    \]
    The first summand of the right hand side vanishes by \eqref{eq:dr_asigma_brho_0}, and the second one vanishes as $\text{d}_{2n-1}(a_\sigma{\sf{b}}_{(2n-4 + \ell)\rho})=0$ by induction using the Leibniz rule. Alternatively, we have
    \[
    \text{d}_{2n-1}(a_\sigma^2{\sf{b}_{(2n-3 + \ell)\rho}})=\text{d}_{2n-1}(a_\sigma{\sf{b}}_0)a_\sigma{\sf{b}}_{(2n-3 + \ell)\rho} + a_\sigma{\sf{b}}_0\text{d}_{2n-1}(a_\sigma{\sf{b}}_{(2n-3 + \ell)\rho}),
    \]
    and note that $\text{d}_{2n-1}(a_\sigma{\sf{b}}_0)=0$. Combining the above sums, we see that any potential value for $\text{d}_{2n-1}(a_\sigma{\sf{b}}_{(2n-3+\ell)\rho})$ is simple $a_\sigma$-torsion. However, the potential target is not simple $a_\sigma$-torsion, so the differential must be trivial. The Leibniz rule, together with the vanishing of all higher differentials on $\overline{\mathrm{D}}$, then implies that all higher differentials vanish as well, completing the proof.
    

\end{proof}

The product structure of the $\Emr_{\infty}$-page of $\textbf{ASliceSS}_{\text{k}\r}(\Bmr_{\Cmr_2}\Sigma_{2})$ can also be determined. To obtain this, we compare the $\text{k}\r$-computation with the $\Hmr\ull{\z}$-computation. The ring structure of $\Hmr\ull{\z}_\star({\Bmr_{\Cmr_2}\Sigma_{2+}})$ is established in \cref{thm:HZHomologyProductStructure}. Since the slice filtration functor $\text{Q}$ is lax monoidal (see \cref{rem: lax monidality}) and $\text{Q}^{-1}(\text{k}\r \wedge {\Bmr_{\Cmr_2}\Sigma_{2+}}) \simeq *$, the projection of $\text{k}\r \wedge {\Bmr_{\Cmr_2}\Sigma_{2+}}$ onto its zeroth slice induces, after passing to homotopy groups, a ring homomorphism 
\begin{equation} \label{eqn: ring map in kR homology compared to HZ-homology}
    \pi_{\star}(\text{k}\r \wedge {\Bmr_{\Cmr_2}\Sigma_{2+}}) \to \pi_{\star} \text{Q}^0 (\text{K}\r \wedge {\Bmr_{\Cmr_2}\Sigma_{2+}}) \to \pi_{\star} \text{Q}_0^0 (\text{K}\r \wedge {\Bmr_{\Cmr_2}\Sigma_{2+}}) \cong  \pi_{\star} (\Hmr\ull{\z} \wedge {\Bmr_{\Cmr_2}\Sigma_{2+}}).
\end{equation}
Since $\sfb_{n \rho + \sigma}, \sfb_{n \rho}$ are permanent cycles for all $n \geq 0$ and remain unchanged throughout the computation, their products are determined already in the 0-slice. Hence, the multiplicative relations among these generators in the 0-slice agree with their relations in the full theory; the only additional structure arises from adjoining the $\overline{v_1}$-tower.

\begin{proposition} \label{thm: ring str of kR BC2C2}
The ring structure on the $\Emr_{\infty}$-page of $\textbf{ASliceSS}_{\text{k}\r}(\Bmr_{\Cmr_2}\Sigma_{2})$, is determined by the following relations:
    \begin{enumerate}
        \item $\sfb_{n \rho+ \sigma} \cdot \sfb_{m \rho + \sigma} = 0$, for $m, n \geq 0$;
        \item $(\alpha \sfb_{n \rho}) \cdot (\beta\sfb_{m \rho}) = (1-\delta_n^m) \alpha\beta\cdot \sfb_{(m+n) \rho}$, where $\delta^m_n$ is the Kronecker delta, and $\alpha,\beta\in \overline{\Kmr}$;
        \item $(\alpha\cdot\sfb_{n\rho}) \cdot \sfb_{m\rho + \sigma} = \alpha\cdot\sfb_{(n+m)\rho + \sigma}$, where $\alpha\in \overline{\Kmr}$.
    \end{enumerate}
\end{proposition}
\begin{proof}
    (1) is straightforward by the previous discussion. Write
    \[
        F: (\text{k}\r)_{\star}({\Bmr_{\Cmr_2}\Sigma_{2+}}) \to \Hmr\ull{\z}_{\star}({\Bmr_{\Cmr_2}\Sigma_{2+}})
    \]
    for the desired ring map in \eqref{eqn: ring map in kR homology compared to HZ-homology}. Because $\text{k}\r$ is connective and $\Hmr\ull{\z}$ is its 0-slice, the map on $\text{RO}(\Cmr_2)$-graded coefficients operates by taking $\overline{v_1} \mapsto 0$. Together with the discussion above, 
    \[
        F(\overline{v_1}) = 0, \quad F(\sfb_{V}) = \sfb_V,
    \]
    where $V = n \rho, n \rho + \sigma$, $m, n \geq 0$. Note that $\overline{\Kmr} = \Hmr_*(\Kmr[\overline{v_1}], \text{d}_3)$. It follows that $F$ acts as the identity on this subalgebra when restricted to $\overline{v_1}^m$ with $m =0$. Therefore, $\Emr_{\infty}(\textbf{ASliceSS}_{\text{k}\r}(\Bmr_{\Cmr_2}\Sigma_{2}))$ inherits the ring structure of $\Hmr\ull{\z}_{\star}({\Bmr_{\Cmr_2}\Sigma_{2+}})$, subject to the corresponding restrictions on the coefficients, yielding (2) and (3).
\end{proof}

\section{Immersions of $\Cmr_2$-Projective Spaces}\label{sec:immersions}

We now apply our K-theory computations to study immersions of $\p(n\rho)$ into multiplies of the regular representation $\rho$. Nonequivariantly, the study of immersions of manifolds in Euclidean spaces is connected to bundle theory by the following theorem of Hirsch:

\begin{thm}[{\cite{HirschImmersions}}]\label{thm:HirschImmersions}
    The tangent bundle functor $\text{T}$ induces a weak homotopy equivalence 
    \[
        \begin{tikzcd}
            \text{Imm}(M, \r^{n+k}) \ar[r, "\text{T}"] & \text{BundMono}(\Tmr M, \Tmr \r^{n+k}),
        \end{tikzcd}
    \]
    where $\text{Imm}(M, \r^{n+k})$ is the space of immersions $M\looparrowright \r^{n+k}$, and $\text{BundMono}(\Tmr M, \Tmr \r^{n+k})$ is the space of bundle monomorphisms. 
\end{thm}

Given an immersion $f:M\looparrowright\r^{n+k}$, it is clear that the differential $df:\Tmr M\to \Tmr \r^{n+k}$ is a bundle monomorphism. If a bundle monomorphism exists, it necessitates the existence of a rank $k$ bundle $\nu$ over $M$ such that
\[
    \Tmr M\oplus \nu\cong \epsilon_{n+k}
\]
as bundles over $M$. Hirsch's theorem allows us to reverse this process. In particular, if $\xi$ is a bundle of rank $k$ over $M$ such that $\Tmr M\oplus \xi\cong \epsilon_{n+k}$, but where $\xi$ is not necessarily the normal bundle associated to an immersion, one can produce an immersion of $M$ into $\r^{n+k}$. If such a bundle exists and given a smooth map $f:M\to \r^{n+k}$, 
\[
    \begin{tikzcd}
        \Tmr M \ar[r, hook, ""] \ar[d, ""] & \Tmr M\oplus \xi \ar[r, "\cong"]\ar[d, ""] & M\times \r^{n+k} \ar[rrr, "(x{,}v)\mapsto (f(x){,} v)"]\ar[d, ""] &&& \r^{n+k}\times \r^{n+k}\ar[d, ""]\\
        M \ar[r, "\operatorname{Id}"] & M \ar[r, "\operatorname{Id}"] & M \ar[rrr, "f"] &&& \r^{n+k}
    \end{tikzcd}
\]
is a bundle monomorphism. If we denote the map of total spaces by $F$, it is almost certain that $df$ is not equal to $F$. However, Hirsch's theorem says that we can always find $g:M\hookrightarrow \r^{n+k}$ so that $dg = F$.

So to produce an immersion $M\looparrowright \r^{n+k}$, it suffices to find a bundle $\xi$, as above. Atiyah then framed finding such a bundle as a problem in K-theory by observing that if $\xi$ is classified by $M\xrightarrow[]{\nu} \operatorname{BO}(k)$, the corresponding stable class obtained by composing with the map $\operatorname{BO}(k)\to \operatorname{BO}$ is a representative for the stable normal bundle for $M$, denoted $-\Tmr M$. For $k$ sufficiently large, one can always immerse $M$ into $\r^{n+k}$ by Whitney's theorem. So it is always possible to find a bundle $\xi$ which represents $-\Tmr M$. 
Atiyah then argued that one can split off trivial summands from such a representative of $-\Tmr M$ to produce smaller rank bundles $\xi$. He defines the smallest such rank, as follows.  

\begin{definition}[\cite{AtiyahImmersions}]
    The \textbf{geometrical dimension} of a (real) virtual bundle $\tau$ over $M$, denoted $\text{gdim}(\tau)$, is the smallest integer $k$ such that $\tau \oplus \epsilon_k$ is equivalent to a non-virtual vector bundle over $M$. 
\end{definition}

When $\tau = -\Tmr M$, the geometrical dimension is bounded above by the dimension of $M$ because $-\Tmr M$ can always be modeled as a vector bundle over $M$ of sufficiently large rank, and one can split off trivial vector bundles from such a model until one reduces the rank to at most the dimension of $M$. The following is then a formal consequence of all these observations.

\begin{thm}[\cite{AtiyahImmersions}]\label{thm:AtiyahImmersion}
    There is an immersion $M\looparrowright \r^{n+k}$ if and only if $\text{gdim}(-\Tmr M)\leq k$. 
\end{thm}

When $M=\r\p^n$, one can obtain an upper bound of $\text{gdim}(-\Tmr M)$ by recalling
\begin{equation}\label{eqn:tangentbundleRP^nandTautological}
    \Tmr \r\p^n\oplus \epsilon_1\cong (n+1)\gamma_1,
\end{equation}
where $\gamma_1$ is the tautological line bundle over $\r\p^n$, and by noting that $[\gamma_1]$ is $2$-torsion:
\begin{thm}[{\cite[Theorem 1]{FujiiKO}}]\label{thm:KO(RP^n)}
    For $n\geq 1$, 
    \[
        \widetilde{\text{KO}}^0(\r\p^n)\cong \zz/2^{\phi(n)}\langle[\gamma_1-\epsilon_1]\rangle,
    \]
    where $\phi(n)$ is the number of integers $s$ such that $0<s\leq n$ and $s\equiv 0, 1, 2, 4\mod{8}$.
\end{thm}

In $\widetilde{\operatorname{KO}}^0(\r\p^n)$,
\[
    [-\Tmr \r\p^n] = [-(n+1)\gamma_1] = [(2^{\ell\cdot\phi(n)}-(n+1))\gamma_1]
\]
by \eqref{eqn:tangentbundleRP^nandTautological} and \Cref{thm:KO(RP^n)}, where $\ell$ is the smallest integer such that $2^{\ell\cdot\phi(n)}-(n+1)>0$. The geometrical dimension of $[-\Tmr \r\p^n]$ is therefore bounded above by $2^{\ell\cdot \phi(n)}-(n+1)$. By \Cref{thm:AtiyahImmersion}, there is an immersion $\r\p^n\to \r^{2^{\ell\cdot \phi(n)} - (n+1)}$.

\subsection{Immersions of $\Gmr$-manifolds} \

Equivariantly, studying immersions into Euclidean space becomes a more delicate problem, as one can now consider immersions of smooth $\Gmr$-manifolds $M$ into $\Gmr$-representations $V$, where $\Gmr$ is any compact Lie group. If $M$ immerses into $V$, then applying $\Hmr$-fixed points for any $\Hmr\subset \Gmr$ yields an immersion $M^\Hmr$ into $V^{\Hmr}$. 
In \cite{Bierstone}, Bierstone identifies the correct class of bundle monomorphisms to study when generalizing \Cref{thm:HirschImmersions}.
Suppose $F:\Tmr M\to \Tmr V$ is a $\Gmr$-equivariant bundle monomorphism, and let $f:M\to V$ be the map induced by $F$. 

Define the subspace
\[
\text{BundMon}_\Gmr^*(\Tmr M, \Tmr V)\subset \text{BundMon}_\Gmr(\Tmr M,\Tmr V)
\]
of $\Gmr$-equivariant bundle monomorphisms $F:\Tmr M\to \Tmr V$ such that $F_x|_{\Tmr(\Gmr x)_x}$ is the induced map from $\Gmr x$ onto $\Gmr f(x)$ for $x \in M$. We refer to such $F$ as a {\bf geometric bundle monomorphisms}.

\begin{thm}\cite{Bierstone}\label{thm:BierstoneImmersion}
    Suppose $\dim M^\Hmr\leq \dim V^\Hmr$ for all $\Hmr\subset \Gmr$. Then the tangent space functor induces a weak homotopy equivalence
    \[
        \begin{tikzcd}
            \text{Imm}_\Gmr(M, V) \ar[r, "\text{T}"] & \text{BundMon}^*_\Gmr(\Tmr M, \Tmr V).
        \end{tikzcd}
    \]
\end{thm}

When $\Gmr$ is a finite discrete group, this simplifies greatly as every bundle monomorphism is geometrical.

\begin{lemma}\label{lem:finiteGroupGeometricalMonoIsAllMono}
    If $\Gmr$ is a finite discrete group, 
    \[
        \text{BundMon}_\Gmr^*(\Tmr M, \Tmr V) = \text{BundMon}_\Gmr(\Tmr M,\Tmr V).
    \]
\end{lemma}

\begin{proof}
    When $\Gmr$ is finite, the tangent bundle of an orbit $\Tmr(\Gmr x)$ is a $\Gmr_x$-indexed product of the zero vector space, where $\Gmr_x$ is the isotropy group of $x\in M$. As such, $\Tmr(\Gmr x)_x = 0$ for all $x\in X$. So any bundle monomorphism $(f,F)$ from $\Tmr M\to \Tmr V$ has $F_x|\Tmr(\Gmr x)_x$ equal to the zero homomorphism and $df|\Tmr(\Gmr x)$ is the zero homomorphism on each fiber. 
\end{proof}

Suppose that $M$ is a smooth $\Gmr$-manifold and $V$ is a $\Gmr$-representation such that $\dim M^\Hmr\leq \dim V^\Hmr$ for all $\Hmr \subset \Gmr$. If $\xi$ is a (non-virtual) $\Gmr$-vector bundle over $M$ of rank $W$ such that $\Tmr M\oplus \xi\cong \epsilon_V$ and $\varphi:M\to V$ is any smooth equivariant map, we can form a bundle monomorphism 
\[
    \begin{tikzcd}
        \Tmr M \ar[r, hook, ""] \ar[d, ""] & \Tmr M\oplus \xi \ar[r, "\cong"]\ar[d, ""] & M\times V \ar[rrr, "(x{,}v)\mapsto (\varphi(x){,} v)"]\ar[d, ""] &&& V\times V\ar[d, ""]\\
        M \ar[r, "\operatorname{Id}"] & M \ar[r, "\operatorname{Id}"] & M \ar[rrr, "\varphi"] &&& V
    \end{tikzcd}
\]
Denote the bundle monomorphism by $(\varphi, \Phi)$. 

\begin{lemma}\label{lem:existenceOfGImmersionIfBundleComplementExists}
    When $\Gmr$ is finite and such a bundle $\xi$, as above, exists, there is a $\Gmr$-equivariant immersion $M\looparrowright V$.
\end{lemma}

\begin{proof}
    By \Cref{lem:finiteGroupGeometricalMonoIsAllMono}, any bundle monomorphism $(\varphi, \Phi)$ is geometrical. The result then follows from \Cref{thm:BierstoneImmersion}.
\end{proof}

In the classical case, it sufficed to find a vector bundle representative $\xi$ of $-[\Tmr M]$, since it guarantees a bundle isomorphism
\[
    \Tmr M \oplus \xi \oplus \epsilon_\ell\cong \epsilon_{n+k}\oplus \epsilon_\ell.
\]
for some $\ell\geq 0$. One then applies the following standard fact from bundle theory.
\begin{thm}[Cancellation]\label{thm:ClassicalCancellation}
    Suppose $\xi$ and $\eta$ are vector bundles over $M$. If $\xi\oplus \epsilon_k\cong \eta\oplus \epsilon_k$ and $\text{rank}(\xi)> \dim(M)+1$, then $\xi\cong \eta$. 
\end{thm}

This theorem is often stated as a corollary of the following result. Let $\zeta$ be a vector bundle over $M$.

\begin{proposition}\label{prop:DecompBundle}
    There exists a bundle isomorphism $\zeta\cong \xi\oplus \epsilon_k$, if $\text{rank}(\zeta) - k > \dim(M)$. Furthermore, this isomorphism is unique up to bundle homotopy equivalence if $\text{rank}(\zeta) - k > \dim(M) + 1$.
\end{proposition}

\begin{proof}
     
     The existence of such an isomorphism is equivalent to the existence of a section of the associated $k$-frame bundle $V_k(\zeta)\to M$. If $\zeta$ has rank $r$, then the fiber of this bundle is $V_k(\r^{r})$. The Stiefel manifold $V_k(\r^r)$ is $(r-k-1)$-connected. One builds a section of this bundle by induction over the skeleta of $M$. The obstructions to extending a section across an $i$-dimensional cell belong to $\Hmr^i(M;\pi_{i-1}(V_k(\r^r))$. Since $\text{rank}(\zeta) - k > \dim(M)$, we have that $\pi_{i-1}(V_k(\mathbb{R}^r))=0$ for $0 < i < \mathrm{dim}(M)$. Thus, these obstruction groups are all trivial. 

    Uniqueness is saying that any two choices of section of the associated $k$-frame bundle are homotopic. A homotopy between any two sections is the same data as a section of the associated $k$-frame bundle $V_k(\zeta)\times [0,1]$ over $M\times [0,1]$, where we obtain the original sections by projecting onto the endpoints of the interval. If we take two sections of the bundle $V_k(\zeta)\to M$, this prescribes the data at the endpoints of such a bundle morphism. Now we need to extend to the ``middle" of the interval. This is exactly the same problem as existence where we need to induct over the cells of $M\times [0,1]$ now instead of just $M$.
\end{proof}

It is unclear to the authors whether a general equivariant analogue of \Cref{thm:ClassicalCancellation} and \Cref{prop:DecompBundle} appear explicitly in the literature. Since we are only concerned with $\Gmr=\Cmr_2$ in our applications, we provide generalizations only in this setting. 

Suppose that $\zeta$ is a $\Cmr_2$-equivariant vector bundle over a $\Cmr_2$-equivariant smooth manifold $M$. Notice that for each $x \in M^{\Cmr_2}$, the fiber $\zeta_x$ carries a linear involution, and therefore decomposes into its $+1$ and $-1$ eigenspaces. These eigenspaces assemble to form a decomposition
\[
    \zeta|_{M^{\Cmr_2}}\cong \zeta^+\oplus \zeta^-
\]
where $\zeta^+$ is the eigenbundle constructed from all of the $+1$-eigenspaces, and $\zeta^-$ is the eigenbundle constructed from all of the $-1$-eigenspaces. 

\begin{proposition}\label{prop:C2DecompBundle}
    There is a $\Cmr_2$-equivariant bundle isomorphism $\zeta\cong \xi\oplus \epsilon_V$ if
    \begin{itemize}
        \item $\text{rank}(\zeta) - \dim(V) > \dim(M)$,
        \item $\text{rank}(\zeta^\pm) - \dim(V^\pm) > \dim(M^{\Cmr_2})$. 
    \end{itemize}
    Furthermore, this bundle isomorphism is unique up to $\Cmr_2$-equivariant bundle homotopy if
    \begin{itemize}
        \item $\text{rank}(\zeta) - \dim(V) > \dim(M) + 1$,
        \item $\text{rank}(\zeta^\pm) - \dim(V^\pm) > \dim(M^{\Cmr_2}) + 1$. 
    \end{itemize}
\end{proposition}

\begin{proof}
    We first prove existence. Because every $\Cmr_2$-representation decomposes as a direct sum of sign and trivial representations, it suffices to study the cases of $V = 1$ and $V = \sigma$. 
    When $V = 1$, a section $M^{\Cmr_2}\to \zeta$ necessarily factors through the fixed points $\zeta^+$. As $\text{rank}(\zeta^+) - 1 > \dim(M^{\Cmr_2})$, such a nonzero section exists by \Cref{prop:DecompBundle}. By the equivariant tubular neighborhood theorem, one can extend this section to a nonzero section in a small neighborhood $N$ of $M^{\Cmr_2}$ in $M$. It remains to extend the section over the free cells of $M$. 
    Free cells in $M$ are of the form $\Cmr_2\times D^i$. So the obstruction to extending over such a cell lies in $\Hmr^i(M; \pi_{i-1}(S^{r-1}))$, which vanishes since $r - 1 > \dim(M) \geq i$.

    The case where $V = \sigma$ follows by the same argument, noting that any section $M^{\Cmr_2}\to \zeta$ necessarily factors through $\zeta^-$. 

    Uniqueness follows from \Cref{prop:DecompBundle}. 
\end{proof}

\begin{thm}[$\Cmr_2$-Cancellation]\label{thm:C2Cancellation}
    Suppose $\xi$ and $\eta$ are $\Cmr_2$-vector bundles over a $\Cmr_2$-smooth manifold $M$. If $\xi\oplus \epsilon_V\cong \eta\oplus \epsilon_V$ and
    \begin{itemize}
        \item $\text{rank}(\xi)>\dim(M) + 1$,
        \item $\text{rank}(\xi^\pm) > \dim(M^{\Cmr_2}) + 1$,
    \end{itemize}
    then $\xi\cong \eta$. 
\end{thm}

\begin{proof}
    Decomposing $V$ into trivial and sign representations, it suffices to study the cases where
    \[
        \xi\oplus \epsilon_1 \cong \eta\oplus \epsilon_1
    \]
    and
    \[
        \xi\oplus \epsilon_\sigma \cong \eta\oplus \epsilon_\sigma.
    \]
    Since the arguments are largely identical in both cases, we present the proof in the former case. Suppose $\xi\oplus \epsilon_1\to \eta\oplus \epsilon_1$ is a bundle isomorphism. A nonvanishing section $s:M\to \epsilon_1 \hookrightarrow \eta\oplus \epsilon_1$ determines a nonvanishing section $M\to \xi\oplus \epsilon_1$. This section defines a rank $1$ trivial subbundle $\tau$ of $\xi\oplus \epsilon_1$, which need not coincide with the canonical summand. Suppose $\tau$ is the trivial summand of $\xi\oplus \epsilon_1$ determined by this section and let $\nu$ denote its complement. By definition, the bundle isomorphism $\xi\oplus \epsilon_1\cong \eta\oplus \epsilon_1$ restricts to a bundle isomorphism $\nu\cong \eta$. Uniqueness of the section of $\xi\oplus \epsilon_1$ up to bundle homotopy implies the section of $\xi\oplus \epsilon_1$ is bundle homotopic to the obvious non-vanishing section $M\to \epsilon_1\to \xi\oplus \epsilon_1$. Such a bundle homotopy defines a bundle isomorphism $\xi\cong \nu$. Consequently, 
    \[
        \xi\cong \nu\cong \eta,
    \]
    establishing the claim.   
\end{proof}

\subsection{Applications to $\p(n\rho)$} \

We are specifically interested in the case when $\Gmr = \Cmr_2$. Just as in the non-equivariant case, one can construct a bundle monomorphism whenever one can find a (non-virtual) vector bundle $\xi$ which is stably equivalent to $-\Tmr M$. One must first clarify what is meant by stable equivalence, since there are two notions of K-theory that may be considered. 

The first is Atiyah's Real K-theory \cite{Atiyah66}, which is represented by the $\Cmr_2$-spectrum $\Kmr\r$. If $X$ is a $\Cmr_2$-space, then $\Kmr\r^0(X)$ is the Grothendieck group of isomorphism classes of virtual rank $0$ Real vector bundles on $X$. If $\tau : X \to X$ denotes the $\Cmr_2$-action, then an Atiyah Real vector bundle on $X$ is a complex vector bundle $\pi : E \to X$ equipped with a $\Cmr_2$-action $\widetilde{\tau} : E \to E$ such that $\pi$ is equivariant and $\widetilde{\tau}$ restricts to a conjugate-linear map on each fiber. 

Alternatively, one can construct a K-theory of $\Cmr_2$-equivariant vector bundles over $X$, represented by a $\Cmr_2$-spectrum $\text{KO}_{\Cmr_2}$. By a $\Cmr_2$-equivariant vector bundle, we mean a real vector bundle $\pi : E \to X$ equipped with a $\Cmr_2$-action $\overline{\tau} : E \to E$ such that $\pi$ is equivariant and $\overline{\tau}$ restricts to an $\mathbb{R}$-linear map on each fiber. 

By forgetting the complex structure of an Atiyah Real vector bundle, one obtains a $\Cmr_2$-equivariant vector bundle. This results in a ring map
\begin{equation}\label{eqn:realificationKR}
    r:\begin{tikzcd}
      \text{K}\r^0(X) \ar[r, ""] & \text{KO}_{\Cmr_2}^0(X)  
    \end{tikzcd}
\end{equation}
sending the equivalence class of an Atiyah Real line bundle to its underlying $\Cmr_2$-equivariant bundle. We refer to this map as \textit{realification}. In particular, if $[\xi]$ is $N$-torsion, it follows that $r([\xi])$ is also $N$-torsion. Supposing $X$ is a $\Cmr_2$-equivariant smooth manifold, if a multiple of $[\Tmr M]\in \text{KO}_{\Cmr_2}^0(X)$ is in the image of $r$, one can study its degree of torsion by studying its lifts in Atiyah Real K-theory. 

Let $X = \p(n\rho)$, and recall from \cite{BWZZ25} that
\begin{equation}\label{eqn:TP(nrho)Relation}
    \Tmr \p(n\rho) \oplus \epsilon_1 \cong n\xi_\rho^{(n)}
\end{equation}
where
\begin{equation*}
    \xi_\rho^{(n)} \coloneqq\left\{ \begin{tikzcd}
        S(n\rho\otimes \tau_2) \times_{\Sigma_2} (\rho \otimes \tau_2)\ar[d, ""]\\
        \p(n\rho)
    \end{tikzcd} \right.
\end{equation*}
is the $\rho$-dimensional tautological bundle over $\p(n\rho)$ analogous to the classical tautological line bundle over $\r\p^{n}$.
This implies that $[\Tmr \p(n\rho)] = [n\xi_\rho^{(n)}]$ in $\text{KO}^0_{\Cmr_2}(\p(n\rho))$. 

Recall the Atiyah Real line bundle is the rank $1$ complex vector bundle
\[
    \xi_\r^{(\infty)} = \left\{\begin{tikzcd}
        S(\c^\infty) \times_{S^1} \c \ar[d, ""]\\
        \Bmr S^1
    \end{tikzcd}\right.
\]
where $\Bmr S^1 \coloneqq S(\c^\infty)/S^1$. 
Let $\xi_\rho^{(\infty)}$ denote the pullback of $\xi_\mathbb{R}^{(\infty)}$ to $\mathrm{B}_{\mathrm{C}_2}\Sigma_2$ via the quotient map
\[
    \begin{tikzcd}
        \Bmr_{\Cmr_2}\Sigma_2 \simeq (\kappa^*S(\c^\infty))/\Sigma_2 \ar[r, "j"] & S(\c^\infty)/S^1 \simeq \Bmr S^1
    \end{tikzcd}
\]
where $\kappa:\Sigma_2\to S^1$ sends the generator of $\Sigma_2$ to $-1$. 

\begin{lemma}\label{lem:realificationAtiyahRealBundle}
    $r([\xi_\r^{(\infty)}]) = [\xi_\rho^{(\infty)}]$.
\end{lemma}

\begin{proof}
    This is immediate as the underlying $\Sigma_2$-representation of $\kappa^*\c$ is isomorphic as an orthogonal representation to $\rho$. 
\end{proof}

Pulling back $\xi_\rho^{(\infty)}$ and $\xi_\r^{(\infty)}$ along the maps $\p(n\rho)\to \p(\infty\rho)$ yields vector bundles $\xi_\rho^{(n)}$ and $\xi_\r^{(n)}$, respectively, which satisfy the evident compatibility relation.

\begin{corollary}
    $r([\xi_\r^{(n)}]) = [\xi_\rho^{(n)}]$.
\end{corollary}
 
By \eqref{eqn:TP(nrho)Relation}, $\Kmr{\r}$-orientability of $\xi_\rho^{(n)}$, and the equivariant Atiyah duality, we have that
\begin{align*}
    \Kmr{\r}^\star (\p(n \rho)_+) & \cong \Kmr{\r}_{-\star} (\Th(-\Tmr \p(n\rho))) \\
        & \cong \Kmr{\r}_{-\star} (\Th(\epsilon_1 - n\xi_{\rho}^{(n)})) \\
        & \cong \Kmr{\r}_{-\star-1} (\Th(-n \xi_{\rho}^{(n)})) \\
        & \cong \Kmr{\r}_{n \rho-\star-1} (\p(n \rho)_+).
\end{align*}
Taking $\star = 0$, the following result gives us the torsion information we need. 
\begin{lemma}
    There is an isomorphism of abelian groups 
    \[
        \Kmr\r^0(\p(n\rho)_+) \cong \zz/2^{n}\langle [\xi_\r^{(n)} - \epsilon_\r]\rangle,
    \]
    where $\epsilon_\r$ is the trivial Atiyah Real line bundle over $\p(n\rho)$. 
\end{lemma}

\begin{proof}
    The above argument supplies us with an isomorphism
    \begin{equation}
    \label{thom iso in proof}
    \Kmr{\r}^0(\p(n\rho)_+) \cong \Kmr{\r}_{n\rho-1}(\p(n\rho)_+).
    \end{equation}
    As \eqref{filtration:p(n rho)} shows that $\Bmr_{\Cmr_2}\Sigma_{2+}$ is filtered by the spaces $\p(n\rho)$, one can determine that the order of the right hand side of \eqref{thom iso in proof} is $2^n$ from our main computation in \Cref{section 3}.
    
    To determine the 2 extensions, recall from the classical computation of $\text{KU}^*(\mathbb{RP}^{2n-1}_+)$ that there is an isomorphism
    \[
    \text{KU}^0(\mathbb{RP}^{2n-1}_+) \cong \mathbb{Z}/2^n.
    \]
    The restriction map
    \[
    \bigoplus_{i=0}^n\mathbb{Z}/2 \cong\Kmr{\r}^0(\p(n\rho)_+) \to \text{KU}^0(\mathbb{RP}^{2n-1}_+) \cong \zz/2^n,
    \]
    maps the $\mathbb{Z}/2$-summand on the $\mathrm{E}_\infty$-page of the $\textbf{ASliceSS}_{\mathrm{k}{\r}}(\Bmr_{\Cmr_2}\Sigma_2)$ isomorphically onto the $\mathbb{Z}/2$-summand on the $\mathrm{E}_\infty$-page of the periodic variant of \eqref{eqn:kuSliceSS}. This implies that the generator of $\text{KU}^0(\mathbb{RP}^\infty_+)$ must pull back to an element in $\Kmr{\r}^0(\mathbb{P}(n\rho)_+)$ of torsion degree at least $2^n$, finishing the proof.
\end{proof}

As a result, $[-n\xi_\rho^{(n)}]$ is equal in $\mathrm{KO}_{\mathrm{C}_2}^0(\mathbb{P}(n\rho)_+)$ to $[(2^{n} - n)\xi_\rho^{(n)}]$ for each $n$. 

\begin{lemma}\label{lem:normalbundlerepofxi_rho}
    There is a bundle isomorphism $n\xi_\rho^{(n)}\oplus (2^n - n)\xi_\rho^{(n)}\cong \epsilon_{2^n\rho}$.
\end{lemma}

\begin{proof}
    We have established 
    \[
        n\xi_\rho^{(n)}\oplus (2^n - n)\xi_\rho^{(n)}\oplus \epsilon_V\cong \epsilon_{2^n\rho}\oplus \epsilon_V
    \]
    for some $\Cmr_2$-representation $V$. As 
    \[
        \text{rank}(2^n\xi_\rho^{(n)}) = 2^{n+1} > (2n - 1) + 1 = \dim(\p(n\rho))+1,
    \]
    and
    \[
        \text{rank}((2^n\xi_\rho^{(n)})^\pm) = 2^n > (n-1) + 1 = \dim(\p(n\rho)^\Cmr_2) + 1,
    \]
    the result follows from \Cref{thm:C2Cancellation}. 
\end{proof}

\begin{corollary}\label{cor:P(nrho)Immersions}
    For every $n\geq 0$, there is a $\Cmr_2$-equivariant immersion $\p(n\rho)\looparrowright 2^n\rho$. 
\end{corollary}

\begin{proof}
    This is immediate from \Cref{lem:existenceOfGImmersionIfBundleComplementExists} with $\xi = (2^n - n)\xi_\rho^{(n)}$.
\end{proof}

\subsection{Immersions via Atiyah Real K-Theory}\label{sec:EfficiencyOfKR} \

The underlying map in \Cref{cor:P(nrho)Immersions} is an immersion $\r\p^{2n-1}\looparrowright \r^{2^{n+1}}$. This is the best possible immersion that one can obtain from the perspective of $\text{KU}$-theory. Explicitly, recall that
\[\widetilde{\text{KU}}^0(\r\p^{2n-1})\cong \zz/2^n\]
is generated by $[\gamma_1^{(2n-1)}\otimes \c - \epsilon_1\otimes \c]$, where $\gamma_1^{(2n-1)}$ is the tautological line bundle over $\r\p^{2n-1}$. The realification map
\[\text{KU}^0(\r\p^{2n-1})\to \text{KO}^0(\r\p^{2n-1})\]
sends the generator in the domain to $2[\gamma_1^{2n-1} - \epsilon_1]$. Without any explicit knowledge of the codomain, we can only conclude that this class is $2^{n}$-torsion. Consequently, \[[-(2n-1)\gamma_1^{(2n-1)}] = [(2^{n+1} - (2n-1))\gamma_1^{(2n-1)}].\] By \Cref{thm:ClassicalCancellation}, this leads to an unstable isomorphism 
\[(2n-1)\gamma_1^{(2n-1)}\oplus (2^{n+1} - (2n-1))\gamma_1^{(2n-1)}\cong 2^{n+1}\epsilon_1,\]
which implies there is an immersion $\r\p^{2n-1} \looparrowright\r^{2^{n+1}}$.

Therefore, one can interpret \Cref{cor:P(nrho)Immersions} as asserting that, from the perspective of constructing equivariant immersions of real projective spaces, Atiyah Real K-theory is as effective as classical $\mathrm{KU}$-theory.

We can reduce the number of copies of the regular representation needed to immerse $\p(n\rho)$ by applying the classical result. The strategy is to consider the restriction-coinduction adjunction
\[
    \begin{tikzcd}
        \Top_*^{\Cmr_2} \ar[rrr, shift right, "\res^{\Cmr_2}_\sfe(-)"'] & & & \Top_* \ar[lll, shift right, "\Map({\Cmr}_2{,} -)"'].
    \end{tikzcd}
\]
Let $\eta:\p(n\rho)\to \Map(\Cmr_2, \r\p^{2n-1})$ be the unit of this adjunction for $\mathbb{P}(n\rho)$. The bundle $\xi_\rho^{(n)}$ can be expressed as a pullback
\begin{equation}\label{eqn:pullbackconiducedBundle}
    \begin{tikzcd}
        \xi_\rho^{(n)} \ar[r, ""] \ar[d, ""] & \Map(\Cmr_2, \gamma_1)\ar[d, ""]\\
        \p(n\rho) \ar[r, ""] & \Map(\Cmr_2, \r\p^{2n-1})
        \arrow["\lrcorner"{anchor=center, pos=0.125}, draw=none, from=1-1, to=2-2]
    \end{tikzcd}
\end{equation}
Coinducing a non-equivariant bundle produces a $\Cmr_2$-equivariant bundle. Since coinduction is an additive functor, we obtain a map of abelian groups
\[
    \text{KO}^0(\r\p^{2n-1}) \to \text{KO}_{\Cmr_2}^0(\Map(\Cmr_2, \r\p^{2n-1})).
\]
These two observations imply that the degree of torsion of the class of $[\xi_\rho^{(n)} - \epsilon_\rho]$ is only slightly smaller than $2^n$. 

\begin{lemma}\label{lem:torsionInKOC2}
    The class $[\xi_\rho^{(n)} - \epsilon_\rho]$ is $2^{\phi(2n-1)}$-torsion, where $\phi$ is as defined in \Cref{thm:KO(RP^n)}. Furthermore, this class is at least $2^{\phi(2n-1)-1}$-torsion.
\end{lemma}

\begin{proof}
    The first claim is a consequence from the previous discussion. For the latter claim, the restriction map $\widetilde{\text{KO}}_{\Cmr_2}^0(\p(n\rho))\to \widetilde{\text{KO}}^0(\r\p^{2n-1})$ sends $[\xi_\rho^{(n)} - \epsilon_\rho]$ to $2[\gamma_1^{(n)} - \epsilon_1]$, so the claim follows from the classical result. 
\end{proof}

Verifying the hypotheses of \Cref{thm:C2Cancellation}, we obtain a more efficient collection of equivariant immersions.

\begin{thm}\label{thm:p(nrho) immersion existence}
    For each $n\geq 1$, there is an $\Cmr_2$-equivariant immersion $\p(n\rho)\looparrowright 2^{\phi(2n-1)}\rho$. 
\end{thm}

\begin{proof}
    This follows from \Cref{lem:torsionInKOC2}, \Cref{lem:existenceOfGImmersionIfBundleComplementExists} and \Cref{thm:BierstoneImmersion}.
\end{proof}

\subsection{Equivariant James Periodicity} \

As a final application, we produce a $\Cmr_2$-equivariant analogue of James periodicity for the following stunted projective spaces.

\begin{definition}\label{defn:stuntedProj}
    Let $n > 1$ and $k\in \zz$. The \textbf{$\Cmr_2$-equivariant stunted projective spectrum} $\p_{k \rho}^{(k+n) \rho}$ is the Thom spectrum
    \[
        \p_{k \rho}^{(k+n) \rho} \coloneqq \Th(k \xi_\rho^{(n)}).
    \]
\end{definition}

\begin{thm}[Equivariant James Periodicity]\label{thm:C2JamesPeriodicity}
    For each $n>1$ and $k\in \z$, there is a stable equivalence
    \[
        \Sigma^{2^{\phi(2n-1)}\rho}\, \p_{k\rho}^{(k+n)\rho} \simeq \p_{(k+2^{\phi(2n-1)})\rho}^{(k + 2^{\phi(2n-1)} + n)\rho} 
    \]
\end{thm}

\begin{proof}
    The Thom spectrum $\Th(k \xi_\rho^{(n)})$ is the homotopy colimit of the diagram
    \[
        \begin{tikzcd}
            S(k\xi_\rho^{(n)}) \ar[r, ""] &  D(k\xi_\rho^{(n)}),
        \end{tikzcd}
    \]
    where $S(k\xi_\rho^{(n)})$ and $D(k\xi_\rho^{(n)})$ are the sphere and disk bundles associated to $k\xi_\rho^{(n)}$, respectively. In particular, this disk bundle is equivariantly equivalent to $\p(n\rho)$. Consequently, the Thom spectrum of $k\xi_\rho^{(n)}$ is determined by the associated spherical fibration $S(k\xi_\rho^{(n)})$.

    One can form a group $\text{J}_{\Cmr_2}(X)$ from $\text{KO}_{\Cmr_2}(X)$ by identifying classes whose associated sphere bundles are equivariantly equivalent. Let $\text{J}_{\Cmr_2}:\text{KO}_{\Cmr_2}(X)\to \text{J}_{\Cmr_2}(X)$ to be the associated quotient map. Composing with the realification function \eqref{eqn:realificationKR}, we obtain a Real $\text{J}$-homomorphism:
    \[
        \text{J}_\r:\text{K}\r^0(X)\to \text{J}_{\Cmr_2}(X).
    \]
    It follows from \Cref{lem:realificationAtiyahRealBundle} that $\text{J}_\r([\xi_\r^{(n)}]) = [S(\xi_\rho^{(n)})]\in \text{J}_{\Cmr_2}(\p(n\rho))$. Consequently, the class of this stable spherical fibration is annihilated by $2^{\phi(2n-1)}$ by \Cref{lem:torsionInKOC2}. 
    
    We obtain the result from the following sequence of stable equivalences:
    \begin{align*}
        \begin{split}
            \p_{(k+2^{\phi(2n-1)})\rho}^{(k + 2^{\phi(2n-1)} + n)\rho} &= \Th((k+2^{\phi(2n-1)})\xi_\rho^{(n)})\\
                &\simeq \Sigma^{2^{\phi(2n-1)}\rho} \Th(k\xi_\rho^{(n)})\\
                &=\Sigma^{2^{\phi(2n-1)}\rho}\, \p_{k\rho}^{(k+n)\rho}.
        \end{split}
    \end{align*}
\end{proof}

\newpage

\begingroup
\raggedright
\bibliography{master}
\bibliographystyle{alpha}
\endgroup

\end{document}